\documentclass{amsart}

\usepackage{subfig, graphicx, amssymb, amsmath}
\newtheorem{theorem}{Theorem}
\newtheorem{proposition}[theorem]{Proposition}
\newtheorem{lemma}[theorem]{Lemma}

\newtheorem{definition}[theorem]{Definition}
\newtheorem{example}[theorem]{Example}
\newtheorem{remark}[theorem]{Remark}

\newcommand{\CC}{\mathbb{C}}
\newcommand{\NN}{\mathbb{N}}
\newcommand{\RR}{\mathbb{R}}

\begin{document}


\title[Fibonacci control system and redundant manipulators]{A Fibonacci control system with application to hyper-redundant manipulators}
\author{Anna Chiara Lai \and Paola Loreti \and Pierluigi Vellucci}
\address{       Dipartimento di Scienze di Base e Applicate per l'Ingegneria \\
              Sapienza Universit\`a di Roma\\
              Via Scarpa 16 - 00181 Roma
}
              \email{anna.lai@sbai.uniroma1.it}           
              \email{paola.loreti@sbai.uniroma1.it}           
              \email{pierluigi.vellucci@sbai.uniroma1.it}           


\begin{abstract} We study a robot snake model based on a discrete linear control system involving Fibonacci sequence and closely related to the theory of expansions
in non-integer bases. The present paper includes an investigation of the reachable workspace,
 a more general analysis of the control system underlying the model, its reachability and local controllability properties and
the relation with expansions in non-integer bases and with iterated function systems.
\end{abstract}

%
\maketitle

\section{Introduction}

The aim of this paper is to give a model of a planar hyper-redundant manipulator, that is analogous in morphology to robotic snakes and  tentacles,
based on a discrete linear dynamical system involving Fibonacci sequence.
This approach is motivated by the ubiquitous presence of Fibonacci numbers in nature (see \cite{Ball99} and \cite{Wi12}) and, in particular, in human limbs \cite{Park03}.

%


The robot proposed in the present paper is a planar manipulator with rigid links and with an arbitrarily large number of degrees of freedom,
i.e., it belongs to the class of so-called \emph{macroscopically-serial hyper-redundant manipulators} -- the term was first introduced in \cite{term}.
The device is controlled by a sequence of couples of discrete actuators on the junctions, ruling both the length of every link and the rotation with respect to the previous link.


Hyper-redundant architecture was intensively studied back to the late 60's, when the first prototype of hyper-redundant robot arm was built \cite{snake}.

The interest of researchers in devices with redundant controls is motivated by their ability to avoid obstacles and to perform new forms of robot locomotion and grasping --
see for instance \cite{Bai86}, \cite{Bur88} and \cite{CB95}.
%

Crucial as it is, effective control of hyper-redundant manipulator is difficult for its redundancy; see, for example \cite{Li04}.
For instance,  the number of points of the reachable workspace increases exponentially with the number of degrees of freedom. In this paper, we employ the self-similarity of Fibonacci sequence in order to provide alternative techniques of investigation of the reachable workspace based on combinatorics and on fractal geometry.

The main purpose of the present paper is to provide a theoretical background suitable for applications to inverse kinematic problems, in a fashion
like \cite{IC96}, where the analysis of the reachable workspace is used to design an algorithm solving the inverse kinematic problem in linear time with respect the number of actuators. Furthermore, in \cite{Kwon91}, the design of a manipulator modeling human arm and with link lengths following the Fibonacci sequence, provides a method for the self-collision avoidance problem. We believe that analogous geometrical properties can be extended to manipulators which are inspired by other biological forms, through the self-similarity induced by Fibonacci numbers.
We motivate the choice of discrete controls via their precision with low cost compared to their continuous counterparts.


Hyper-redundant manipulators considered here are planar manipulators. This is only a first step in exploring an approach that, to the best of our knowledge, could add novelty to the existing literature in this field; therefore, for future work, its extension to the three-dimensional case represents a natural progress of this paper.

We finally anticipate to the reader that the workspaces of planar manipulators of above cited papers (e.g. \cite{IC96}) are quite different from those depicted here. This
is mainly due to the fact that we represent only a subset of the workspace, corresponding to the particular subclass of \emph{full-rotation configurations} whose relation with fractal geometry is the most striking.
Furthermore, unlike above mentioned works, our robotic device has a telescopic structure modeled by the possibility of ruling not only the angle between but also the length of each link: we believe this additional feature to possibly affect the shape of the workspace.

\vskip0.5cm
The theoretical background relies on the theory of Iterated Function Systems  -- see \cite{Fal90} for a general introduction on the topic. The approach proposed here is inspired by the relation between
robotics and theory of expansions in non-integer bases, that was first introduced in \cite{CP01} and later applied to planar manipulators in \cite{LL11}, \cite{LL11a},\cite{LL12} and \cite{LLV14}).
 For an overview on the expansions in non-integer bases we refer to the R\`enyi's seminal paper \cite{Ren57} and to  the papers \cite{Par60} and \cite{EK98}.
 For the geometrical aspects of the expansions in complex base, namely the arguments that are more related to problem studied here, we refer to the papers \cite{Gil81},\cite{Gil87},\cite{IKR92} and to \cite{KL07}. The techniques developed in the present paper in order to study the full-rotation configuration
generalize previous results in \cite{Lai11}.


\subsection{Brief description of the main results}

A discrete dynamical system models the position of the extremal
junction of the manipulator. The model includes two binary control parameters on every link. The first control parameter, denoted by $u_n$, rules the length of
the $n$-th link $l_n:=u_n f_n q^{-n}$, where $f_n$ is the $n$-th Fibonacci number and $q$ is a constant scaling ratio, corresponding to the choice of $u_n=0$ and
$u_n=1$, respectively. The
other control, $v_n$, rules the angle between the current link and the
previous one, denoted by $\omega_n:=(\pi-\omega) v_n$, where $\omega$ a fixed angle in $(0,\pi)$. Therefore when $v_n=0$,  the $n$-th link
is collinear with the previous, and when $v_n=1$, it forms a fixed angle $\pi-\omega\in(0,\pi)$ with the $n-1$-th link.  In Section \ref{smodel} we show that, under these assumptions, the position of the $n$-th junction, $x_n(\mathbf u,\mathbf v)$ is ruled by the relation
\begin{equation}
x_{n}(\mathbf u,\mathbf v)=x_{n-1}+u_n\frac{f_n}{q^n}e^{-i\omega \sum_{h=0}^{n} v_h}
\end{equation}
where $\mathbf u=(u_j),~\mathbf v=(v_j)\in\{0,1\}^\infty$.  By assuming that the $n$-th junction is positioned at time $n$ (namely by reading the index $n$ as a discrete time variable) above equation may be reinterpreted as a discrete control system, whose trajectories model the configurations of the manipulator. This is a stationary problem: indeed, at this stage of the investigation we are interested on the reachable workspace of the manipulator (namely a static feature of robot) rather than its kinematics. In this setting, if the number of the links is finite, say it is equal to $N$, then the position of the end effector of the manipulator (i.e., the position of its extremal junction) is represented
by $x_N(\mathbf u,\mathbf v)$. We call \emph{reachable workspace} the set
$$W_{N,q,\omega}:=\{x_N(\mathbf u,\mathbf v)\mid \mathbf u,\mathbf v\in\{0,1\}^N\}.$$
By allowing an infinite number of the links, we also may introduce the definition of \emph{asymptotic reachable workspace}
$$W_{\infty,q,\omega}:=\{\lim_{N\to\infty} x_N(\mathbf u,\mathbf v)\mid \mathbf u,\mathbf v\in\{0,1\}^\infty\}.$$

The first main results, Theorem \ref{L} and Theorem \ref{workspace}, deal with some asymptotic controllability properties of the manipulator.

  Indeed, the investigation begins with the study of the quantity
 $$L(\mathbf u):=\sum_{n=1}^\infty \frac{u_n f_n}{q^n}$$
 namely of the \emph{total length} of the manipulator\footnote{Notice that $L(\mathbf u)=L(\mathbf u,\mathbf v)$ for all $\mathbf v\in \{0,1\}^\infty$ where $L(\mathbf u,\mathbf v):=\sum_{n=1}^\infty |x_n(\mathbf u,\mathbf v)-x_{n-1}(\mathbf u,\mathbf v)|$}.
 First  of all we notice that the condition $q>\varphi$, where $\varphi=(1+\sqrt{5})/2$ is the Golden Ratio, ensures the convergence of above series.
 Theorem \ref{L} is a first investigation of the behaviour of the set of possible total lengths
 $$L_{\infty,q}:=\{L(\mathbf u)\mid \mathbf u\in\{0,1\}^\infty\}$$
 as $q\to \infty$. In particular we show that if $q$ is lower or equal to the value $1+\sqrt{3}$ then $L_{\infty,q}$ is an interval. This estimate is sharp,
 indeed we shall also prove that when $q>1+\sqrt{3}$ then $L_{\infty,q}$ is a disconnected set.
  In other words, Theorem \ref{L} states that  we can arbitrarily set the length of manipulator within the range $[0,L(\mathbf 1)]$ (where we have set $\mathbf 1:=(1,1,\dots,1,...)$)
 if and only if the scaling ratio $q$ belongs to the range $(\varphi, 1+\sqrt{3}]$. The proof of Theorem \ref{L} is constructive and an explicit algorithm is given.

 Theorem \ref{L} turns out to be also a useful tool in order to prove sufficient conditions for the \emph{local asymptotic controllability} of the control system underlying the model (see Theorem \ref{workspace}), that is the possibility of place the end effector of the manipulator arbitrarily close to any point belonging to a sufficiently small neighborhood of the origin.
 More precisely, Theorem \ref{workspace} states that, under some technical assumptions (namely we assume the that the maximal rotation angle $\omega$ is of the form $2d\pi/p$ for some $d,p\in\NN$), if $q$ 
   belongs to a certain range, then the asymptotic reachable workspace contains a neighborhood of the origin\footnote{Actually, we prove that such a neighborhood is indeed a polygon which is symmetric with respect to the origin.}.

 The approach in the investigation of $L_{\infty,q}$ and $R_{\infty,q,\omega}$,  the latter defined as
 $$R_{\infty,q,\omega}:=\left\{\sum_{k=0}^\infty u_k \frac{f_k}{q^ke^{i\omega k}} \mid \mathbf u\in\{0,1\}^\infty\right\},$$
  strongly relies on the particular choice of the lengths of the links, $l_n(u_n):=u_n f_n q^{-n}$, and in particular, on the fact that, fixing $\mathbf u=(u_n)$ the ''backward'' sequence $\bar L_n( \mathbf u)=\sum_{j=1}^n l_j(u_{n-j})$ satisfies the recursive, contractive relation
 \begin{equation}\label{rec}
 \bar L_{n+1}(\mathbf u)=\frac{u_{n}+\bar L_n(\mathbf u)}{q}+\frac{\bar L_{n-1}(\mathbf u)}{q^2}.
 \end{equation}
%

A suitable generalization of (\ref{rec}) is interpreted as a discrete control dynamical system, the \emph{Fibonacci control system}, which is investigated by means of combinatorial arguments.

    We then use a generalization of above approach in order to study a suitable subset of $R_{\infty,q,\omega}$, the set of \emph{full-rotation configurations} (namely the configurations corresponding to the choice $\mathbf v=\mathbf 1$).
   This approach is motivated by the fact that the full-rotation configurations satisfy a contractive, recursive relation similar to (\ref{rec}).

  The third main result of the present paper concerns the characterization of $L_{\infty,q,\omega}$ and the set of full-rotation configurations in terms of the attractor of a suitable Iterated Function System (IFS). This approach gives access to well-established results in fractal geometry in order to further investigate the topological properties of
  the reachable workspace, and to use known efficient algorithms for the generation of self-similar sets (e.g. Random Iteration Algorithm) to have a numerical approximation of the asymptotic reachable set.

 In what follows we show some numerical simulations approximating the asymptotic reachable set associated with full-rotation configurations. However a deeper exploiting of these potential applications is beyond the purposes of present work.

  \vskip0.5cm
  We finally remark that for all $N\geq 0$ we have the inclusion $W_{N,q,\omega}\subset W_{\infty,q,\omega}$ and, consequently, the Hausdorff distance between $W_{N,q,\omega}$ and $W_{\infty,q,\omega}$ satisfies
   \begin{align*}
 d_H(W_{N,q,\omega},W_{\infty,q,\omega})&=\sup_{x_\infty\in W_{\infty,q,\omega}} \inf_{x_N\in W_{N,q,\omega}} |x_\infty-x_N|\\
 &\leq \sum_{k=N+1}^\infty \frac{f_{k}}{q^{k}}\leq\frac{q}{q^N(q^2-q-1)}.
   \end{align*}
  Above relation establishes a global error estimate for the approximation of $W_{\infty,q,\omega}$ with $W_{N,q,\omega}$, hence every above mentioned asymptotic controllability property is inherited by a practical implementable manipulator with a finite number of links $N$ by paying an explicitly given, exponential decaying cost in terms of precision.

\subsection{Organization of the paper}
In Section \ref{smodel} we introduce the model and we state the main results on the density of the reachable workspace. The remaining part of the paper is devoted to the analysis of the
dynamical system underlying the model. Section \ref{ssectioncontrol} is devoted to the introduction of such Fibonacci control system and to its preliminary properties. In Section \ref{sreachability} and Section \ref{slocal}
we establish some properties of reachability and local controllability.  Finally in Section \ref{sifs} we establish a relation with the theory of Iterated Function Systems and we point out some parallelisms with
classical expansions in non-integer bases.

\tableofcontents
\section{A model for a snake-like robot.}\label{smodel}
Throughout this section we introduce a model for a snake-like robot.
We assume links and junctions to be thin, so to be respectively approximate with their middle axes and barycentres.
 We also assume axes and barycentres to be coplanar and, by employing the isometry between $\RR^2$, we use the symbols $x_0,x_1,..., x_n\in \CC$
 to denote the position of the barycentres of the junctions, therefore the length $l_n$ of the $n$-th link is
\begin{equation}\label{l}
 l_n=|x_{n}-x_{n-1}|
\end{equation}

We assume $l_n$ to be ruled by a binary control $u_n$, and in particular,
\begin{equation}\label{ln}
l_n:=u_n\frac{f_n}{q^n}.
\end{equation}
where $(f_n)$ is Fibonacci sequence, namely $f_0=f_1:=1$ and $f_{n+2}=f_{n+1}+f_n$ for all $n\geq 0$.

Now, consider the quantity
$$L(\mathbf u)=\sum_{n=0}^\infty l_n(u_n)\quad \text{with } \mathbf u=(u_n)\in\{0,1\}^\infty$$
representing the total length of the configuration of the snake-like robot corresponding to the control $\mathbf u$.
\begin{remark}
 In order to simplify subsequent notations we shall fix as the base of the manipulator the point $x_{-1}=0$, so that the $0$-th link is well defined and it may be of length either $0$ or $1$.
\end{remark}

 We shall also use the quantity
\begin{equation}\label{defS}
S(q,h,p):=\sum_{k=0}^\infty\frac{f_{pk+h}}{q^{pk}}.
\end{equation}


The most general form of this definition will be used only in Section \ref{slocal}. At this stage, it is useful to introduce for brevity the notation
\begin{equation}\label{s1def}
S(q):=S(q,0,1)= \sum_{n=0}^\infty \frac{f_{n}}{q^n}= \begin{cases}
           \displaystyle{ \frac{q^2}{q^2-q-1}} \quad &\text{ if }q>\varphi;\\\\
            +\infty \quad &\text{ if }  q\in (0,\varphi]
           \end{cases}
\end{equation}
where $\varphi:=\frac{1+\sqrt{5}}{2}$ denotes the Golden Mean.

\begin{remark}
 If $q>\varphi$ then for every $\mathbf u\in\{0,1\}^\infty$, one has $L(\mathbf u)\in[L(\mathbf 0), L(\mathbf 1)]=[0, S(q)].$
\end{remark}

In what follows we show that if the scaling ratio $q$ belongs to a fixed interval and if we allow the number of links to be infinite, then we may constraint the total length of
the snake-like robot $L(\mathbf u)$ to be any value in the interval $[0,S(q)]$.
\begin{theorem}\label{L}
 If $q\in(\varphi,1+\sqrt{3}]$ then for every $\bar L\in [0,S(q)]$ there exists a binary control sequence $\mathbf u\in\{0,1\}^\infty$ such that
$$L(\mathbf u)=\bar L.$$
\end{theorem}
\begin{remark}
The proof of Theorem \ref{L} is postponed  to Section \ref{sreachability} below.
\end{remark}


We now continue the building of the model. In view of (\ref{l}), if $x_0=0$ one has for every $n$
\begin{equation}\label{xn}
x_n(\mathbf u)=\sum_{k=0}^n u_k \frac{f_k}{q^ke^{i\omega_k}},
 \end{equation}
where $-\omega_{k}\in(-\pi,\pi]$ is the argument of $x_{k}-x_{k-1}$ for $k=1,\dots,n$ and, consequently,
it represents the orientation of the $k$-th link with respect to the global reference system given by the real and imaginary axes.

\begin{example}
If the angle between two consecutive links is constantly equal to $\pi-\omega\in[0,2\pi)$, then $\omega_n=n\omega \mod (-\pi,\pi]$.
\end{example}

So far we introduced a control sequence ruling the length of each link. We now endow the model with another binary control sequence $\mathbf v=(v_n)$, ruling the angle between two consecutive links.
 In the model, the angle between two consecutive links is either $\pi$ or $\pi-\omega$ for some fixed $\omega \in (0,\pi)$.
 If $v_n = 0$ then the angle between the $n-1$-th link and the $n$-th link is $\pi$, while if $v_n = 1$
then the angle between the $n-1$-th link and the $n$-th link is $\pi-\omega$ so that
\begin{equation}
v_n=
 \begin{cases}
  1 \ \ \emph{rotation of the angle $\omega$ of the n-th link;}\\
  0 \ \ \emph{no rotation.}
 \end{cases}
\end{equation}

We notice that, under these assumptions, $\omega_n=\omega_n(\mathbf v)$ in (\ref{xn}) is indeed a controlled quantity, while $L(\mathbf u)$ is yet independent from $\mathbf v$.

\begin{proposition}\label{pomega}
Let $n \geq 0$ and $u_j=1$ and $v_j \in \{0,1\}$ for $j = 1,...,n$. Then
\begin{equation}
\label{omega}
\omega_n=\sum_{j=1}^{n}v_j \omega~\mod (-\pi,\pi]
\end{equation}
\end{proposition}
\begin{proof}
We adopt the notation $Arg(z)\in(-\pi,\pi]$ to represent the principal value of the argument function $arg(z)$. In view of (\ref{xn})
\begin{equation}\label{wn}
w_{n+1} = -\text{Arg}(x_{n+1}(\mathbf u)-x_{n}(\mathbf u)).
\end{equation}
On the other hand,  $x_n$ is the vertex of the angle between the $n$-th link and the $n + 1$-th link, therefore we have the relations
\begin{equation}
\label{arg2}
\text{Arg}(x_{n+1}(\mathbf u) - x_n (\mathbf u)) - \text{Arg}(x_{n-1}( \mathbf u) - x_n(\mathbf u))~ \mod(-\pi,\pi]= -v_{n+1}\omega
\end{equation}
By a comparison between (\ref{wn}) and (\ref{arg2}) we get
\begin{equation}
w_{n+1}= w_n+v_{n+1}\omega~\mod(-\pi,\pi].
\end{equation}
and, consequently, the claim.
\end{proof}

\begin{remark}
We notice that if $u_n=0$ then any choice of $\omega_n(\mathbf v)$ satisfies \ref{xn}.
So, if the link is not extended,
 the rotation of the angle is meant as a rotation of the
 reference frame of the link.

 For example, if $v_{n}=v_{n+1}=u_{n-1}=u_{n+1}=1$ and $u_n=0$, one has that $x_{n-1}=x_{n}$ but the angle formed by the $n-1$-th junction and the $n+1$-th junction
is $\pi-2\omega$.
\end{remark}

In view of Proposition \ref{pomega}  and of above Remark, we set $\omega_n(\mathbf v):=\sum_{j=0}^n v_j\omega$, so that the complete control system for the joints of manipulator reads:
\begin{equation}
\label{eq:pier}
x_n(\mathbf u,\mathbf v)=
\sum_{k=0}^n u_k\frac{f_k}{q^k }e^{-i\omega \sum_{j=0}^{k}v_{j}}.
\end{equation}

The second main result describes the topology of the asymptotic reachable workspace when the rotation angle $\omega$ is rational with respect to $\pi$, namely it satisfies
$\omega=2\pi \frac{d}{p}$ for some $d,p\in \NN$. One has a local controllability result when the scaling ratio $q$ is lower than a threshold depending on $p$, that we denote $q(p)$. In particular
$q(p)$ is defined as the greatest real solution of the equation
$$\sum_{k=0}^\infty\frac{f_{pk}}{q^{pk}}=2.$$
In Section \ref{sqp} below we give a closed formula for $q(p)$.
\begin{theorem}\label{workspace}
 If $\omega=2\pi\frac{d}{p}$ for some $d,p\in \NN$ and if $q\in (\varphi, q(p)]$ then the asymptotic reachable workspace
$$W_{\infty,q,\omega}:=\left\{\lim_{n\to\infty} x_n(\mathbf u,\mathbf v)\mid \mathbf u,\mathbf v\in\{0,1\}^\infty\right\}$$
contains a neighborhood of the origin.
\end{theorem}
The proof of Theorem \ref{workspace} is postponed to Section \ref{proofworkspace} below.
\section{A Fibonacci control system}\label{ssectioncontrol}
Throughout this section we introduce an auxiliary control system, that we call \emph{Fibonacci control system} and we study its asymptotic reachable set.

We shall see that the reachability properties of the Fibonacci control system are
somehow inherited by manipulator (modeled in previous section as the sequence of junctions $x(\mathbf u,\mathbf v)$) and that this relation provides an indirect proof of Theorem \ref{L} and Theorem \ref{workspace}.

In order to gradually introduce Fibonacci control system, we begin with some remarks on particular configurations of $x(\mathbf u,\mathbf v)$.

We notice that for every $\mathbf u$
$$x(\mathbf u,\mathbf 0)=\sum_{k=0}^\infty u_k \frac{f_k}{q^k}=L(\mathbf u)$$
and
$$x(\mathbf u,\mathbf 1)=\sum_{k=0}^\infty u_k \frac{f_k}{q^ke^{i\omega k}}=\sum_{k=0}^\infty u_k \frac{f_k}{z^k}, \quad \text{where } z=q e^{i\omega}.$$
Then both Theorem \ref{L} and Theorem \ref{workspace} are related to the study of the set
$$R_\infty(z):=\left\{\sum_{k=0}^\infty u_k \frac{f_k}{z^k}\mid u_k\in\{0,1\}\right\}.$$

Indeed
$$L_\infty(q)=\{L(\mathbf u)\mid \mathbf u\in\{0,1\}^\infty\}=R_\infty(q)$$  and
$$W_{\infty,q,\omega}\supseteq \{x(\mathbf u,\mathbf 1)\mid \mathbf u\in\{0,1\}^\infty\}=R_\infty(qe^{i\omega})$$
In particular,  the relation with Theorem \ref{workspace} becomes clear by noticing that
 if we are able to show that $R_\infty(qe^{i\omega})$ is a neighborhood of the origin then the claim of Theorem \ref{workspace} follows.

 \begin{remark}\label{top}
  Notice that if $|z|>\varphi$ then  $R(z)$ is well defined and it is a compact set. Indeed one has
  \begin{align*}
  \lim_{n\to\infty}|\sum_{k=n}^\infty u_k \frac{f_k}{z^k}|\leq \lim_{n\to\infty}\sum_{k=0}^n | \frac{f_k}{z^k}|\leq \lim_{n\to\infty}\sum_{k=n}^\infty  \frac{\varphi^{k-1}}{|z|^k}=0.
\end{align*}
 (for the proof of the estimate $f_{k}\leq \varphi^{k-1}$ see Proposition \ref{Fnphi} below) and, consequently, the convergence of the series  $\sum_{k=0}^\infty u_k \frac{f_k}{z^k}$. Furthermore one has
$$|\sum_{k=0}^\infty u_k \frac{f_k}{z^k}|\leq \varphi^{-1}\left(1+\frac{1}{1-\varphi/|z|}\right)$$
thus $R(z)$ is a bounded set. Finally $R(z)$ is closed
by the continuity of the map
$$\mathbf u\mapsto \sum_{k=0}^\infty u_k \frac{f_k}{z^k}$$
with respect to the topology on infinite sequences induced by the distance $d(\mathbf u,\mathbf v)=2^{-\min\{k \mid u_k\not= v_k\}}$.
\end{remark}

In view of above reasoning, in what follows we shall focus on the study of $R_\infty(z)$, by constructing the theoretical background necessary to prove Theorem \ref{L} and Theorem \ref{workspace}
and by investigating further properties of $R_\infty(z)$.

We finally introduce the Fibonacci control system
\begin{equation}\label{F}\tag{F}
 \begin{cases}
  \bar x_0=u_0\\
  \bar x_1=u_1+\frac{u_0}{z}\\
  \bar x_{n+2}=u_{n+2}+\frac{\bar x_{n+1}}{z}+\frac{\bar x_{n}}{z^2}.
 \end{cases}
\end{equation}
and we denote by $x_n(\mathbf  u)$ the (discrete) trajectory corresponding to the control $\mathbf u\in\{0,1\}$.
We show that $R_\infty(z)$ is the asymptotic reachable set of (F).
\begin{proposition}
Let $z\in\CC$ be such that $|z|>\varphi$, where $\varphi=(1+\sqrt{5})/2$ is the Golden Ratio. Then $x\in R_\infty(z)$ if and only if
$$x=\lim_{n\to \infty} \bar x_n(\mathbf u)$$
 for some $\mathbf u\in \{0,1\}^\infty$.
\end{proposition}
\begin{proof}
In view of Remark \ref{top} and, in particular, of the convergence of the series
$$\sum_{k=0}^{\infty}\frac{f_k}{z^k} u_{k},$$ it suffices to show by induction on $n$ the equality
 $$\bar x_{n+2}(\mathbf u)=\sum_{k=0}^{n+2}\frac{f_k}{z^k} u_{n+2-k}.$$
 The case $n=0$ follows by direct computation. Assume now as inductive hypothesis
$$\bar x_n=\sum_{k=0}^n\frac{f_k}{z^k} u_{n-k}\qquad\text{and}\qquad \bar x_{n+1}=\sum_{k=0}^{n+1}\frac{f_k}{z^k} u_{n+1-k}$$
so that
$$\frac{\bar x_n}{z^2}=\sum_{k=0}^n\frac{f_k}{z^{k+2}} u_{n-k} \quad \text{and}\quad \frac{\bar x_{n+1}}{z}=\frac{f_0}{z}u_{n+1} +\sum_{k=0}^n\frac{f_{k+1}}{z^{k+2}} u_{n-k}$$
and, finally,
\begin{align*}
 \bar x_{n+2}&=u_{n+2}+\frac{\bar x_{n+1}}{z}+\frac{\bar x_{n}}{z^2}=u_{n+2}+\frac{f_0}{q}u_{n+1}+\sum_{k=2}^{n+2}\frac{f_k}{z^k}u_{n+2-k}\\&\stackrel{f_0=f_1}{=}\sum_{k=0}^{n+2}\frac{f_k}{z^k} u_{n+2-k}.
\end{align*}
\end{proof}
\subsection{Asymptotical reachable set  in real case}\label{sreachability}
Throughout this section we consider a real number $q>\varphi$ and we show that $R_\infty(q)=[0,S(q)]$ if and only if $q\leq 1+\sqrt{3}$ (namely we prove Theorem \ref{L}\footnote{Indeed the claim immediately follows by recalling the equality $\{L(\mathbf u)\mid \mathbf u\in\{0,1\}^\infty\}=R_\infty(q)$}) and we give a greedy algorithm steering any reachable point of (F) to the origin.
For brevity, we specialize the definition of $S(q,h,p)$ given in (\ref{defS}) as follows:
\begin{equation}\label{S1}
S(q,h):=\sum_{k=0}^\infty \frac{f_{h+k}}{q^k}=\frac{q^2 f_h+q f_{h-1}}{q^2-q-1}
\end{equation}

Last equality can be proved by a simple inductive argument.  We also shall use the following recursive relation
\begin{equation}\label{S2}
S(q,h)= q (S(q,h-1) - f_{h-1}).
\end{equation}
Finally note that $S(q,0)=S(q)$.
\begin{lemma}\label{lbound}
 Let $q>\varphi$. For every $h$
\begin{equation}\label{estimate}
 f_h\leq \frac{S(q,h+1)}{q}
\end{equation}
if and only if $q\leq 1+\sqrt{3}$.
\end{lemma}
\begin{proof}
 First of all note that for every $h$
$$\frac{f_{h+1}}{f_h}\geq 1=\frac{f_1}{f_0}$$
and consequently $q\leq 1+\sqrt{3}$ if and only if
$$q\leq \frac{1}{2}\left(\frac{f_{h+1}}{f_h}+1\right) + \sqrt{\frac{1}{4}\left(\frac{f_{h+1}}{f_h}+1\right)^2+2}\quad \text{ for every $h$}.$$
This, together with $q>\varphi>0$ implies that $q\leq 1+\sqrt{3}$ is equivalent to
$$f_h\leq \frac{q f_{h+1}+f_h}{q^2-q-1}\left(=\frac{S(q,h+1)}{q}\right)\quad \text{ for every $h$}.$$
\end{proof}

\begin{theorem}\label{reachability}
 Let $q\leq 1+\sqrt{3}$ and $x\in[0,S(q,0)]$ and consider the sequences $(r_h)$ and $(u_h)$ defined by
\begin{equation}\label{rh}
 \begin{cases}
  r_0=x;\\
  u_h=\begin{cases}
1 &\text{if } r_h\in[f_h,S(q,h)]\\
0 &\text{otherwise }
 \end{cases}\\
r_{h+1}=q(r_h-u_h f_h)
\end{cases}
\end{equation}
Then
\begin{equation}\label{x}
x=\sum_{k=0}^\infty \frac{f_k}{q^k} u_k
\end{equation}
and, consequently, $R_\infty(q)=[0,S(q,0)]$.
Moreover if $q>1+\sqrt{3}$ then $R_\infty \subsetneq [0,S(q,0)]$.
\end{theorem}

\begin{proof}
 Fix $x\in[0,S(q,0)]$ and first of all note that
\begin{equation}\label{remainder}
 x=\sum_{k=0}^h \frac{f_k}{q^k}u_k + \frac{r_{h+1}}{q^{h+1}} \quad \text{ for all } h\geq 0.
\end{equation}
 Indeed  above equality can be shown by induction on $h$. For $h=0$ one has $r_1=q(x-u_0f_0)$ and consequently $x=f_0u_0+r_1/q$. Assume now (\ref{remainder}) as inductive hypothesis. Then
$$r_{h+2}=q^{h+2} \left( x- \sum_{k=0}^h \frac{f_k}{q^k}u_k\right) - q f_{h+1}u_{h+1}$$
and, consequently,
$$x=\sum_{k=0}^{h+1} \frac{f_k}{q^k}u_k + \frac{r_{h+2}}{q^{h+2}}.$$

Now we claim that  if $q\leq 1+\sqrt{3}$ then
\begin{equation}\label{bound}
 r_h\in[0,S(q,h)]\quad \text{for every $h$}.
\end{equation}
We show the above inclusion by induction. If $h=0$ then the claim follows by the definition of $r_0$ and by the fact that $x\in[0,S(q,0)]$. Assume now (\ref{bound}) as inductive hypothesis.
One has $r_h\in[0,S(q,h)]=[0, f_h)\cup[f_h,S(q,h)]$. If $r_h\in [0,f_h)$ then $r_{h+1}=q r_h\in [0,q f_h]\subseteq [0, S(q,h+1)]$ - where the last inclusion follows by Lemma \ref{lbound}.
If otherwise $r_h\in [f_h,S(q,h)]$ then $r_{h+1}=q(r_h-f_h)\subseteq [0, q(S(q,h)-f_h)]=[0, S(q,h+1)]$ - see (\ref{S2}).

Recalling $f_n\sim \varphi^n$ as $n\to \infty$,  one has
\begin{align*}
\sum_{k=0}^\infty \frac{f_k}{q^k}u_k&=\lim_{h\to\infty} \sum_{k=0}^{h-1} \frac{f_k}{q^k}u_k\stackrel{(\ref{remainder})}{=} x-\lim_{h\to\infty}\frac{r_h}{q^h}
\stackrel{(\ref{bound})}{\geq} x-\lim_{h\to\infty} \frac{S(q,h)}{q^h}\\&\stackrel{(\ref{S1})}{=} x-\lim_{h\to\infty} \frac{q^2 f_{h+1}+ q f_h}{q^h(q^2-q-1)}=x.
\end{align*}
On the other hand
$$\sum_{k=0}^\infty \frac{f_k}{q^k}u_k=x-\lim_{h\to\infty} \frac{r_h}{q^h}\leq x$$
and this proves (\ref{x}).
It follows by the arbitrariness of $x$ that if $q\leq 1+\sqrt{3}$ then $R_\infty=[0,S(q,0)]$.

Finally assume $q>1+\sqrt{3}$. By Lemma \ref{lbound} there exists $x\in (S(q,1)/q,f_1)$. In order to find a contradiction, assume $x\in R_\infty$. Then
$$x=u_0f_0+\frac{1}{q}\sum_{k=0}^\infty \frac{f_{k+1}}{q^k}u_{k+1}$$
Note that $u_0\not=1$ because $x<f_1=1$. Then $u_0=0$ and
$$x= \frac{1}{q}\sum_{k=0}^\infty \frac{f_{k+1}}{q^k}u_{k+1}\leq \frac{1}{q}\sum_{k=0}^\infty \frac{f_{k+1}}{q^k}=\frac{S(q,1)}{q}$$
but this contradicts $x\in (S(1,q)/q,f_1)$. Then $x\in [0,S(q,0)]\setminus R_\infty$ and this concludes the proof.
\end{proof}

\subsection{Asymptotical reachable set in complex case}\label{slocal}
Throughtout this section we investigate $R_\infty(z)$ with
 $z = q e^{i \omega}$ and $ \omega=\frac{d}{p}2\pi;~ d,p\in \NN.$
First of all we notice that $z^p=q^p$ and consequently
\begin{equation}\label{split}
 \sum_{k=0}^\infty u_k\frac{f_k}{z^k}=\sum_{h=0}^{p-1} z^{-h}\sum_{k=0}^\infty u_{pk+h}\frac{f_{pk+h}}{q^{pk}}.
\end{equation}
Above equality implies that if $p\geq 2$ and if
$$ R^h_\infty:=\left\{ \sum_{k=0}^\infty \frac{u_{pk+h}f_{pk+h}}{q^{pk}}\mid u_{pk+h}\in \{0,1\}\right\}$$
is an interval (and not a disconnected set) then
$$R_\infty(z)=\left\{\sum_{k=0}^\infty \frac{f_k}{z^k}u_k\mid u_k\in\{0,1\}\right\}=\sum_{j=0}^{p-1} z^{-h} R_\infty^h$$
is a polygon containing the origin in its interior - note that $\min R^h_\infty =0$.
In what follows we show that if $q$ is small enough, then such a local controllability condition is satisfied.

By definition \ref{defS}, so that $R^h_\infty\subset [0,S(q,h,p)]$ for every $h=0,\dots,p-1$ and from simple inductive arguments, we have the following recursive relation
\begin{equation}\label{Shp1}
 S(q,h,p)=f_{h-1} S(q,1,p)+ f_{h-2} S(q,0,p)
\end{equation}
Moreover one has
\begin{equation}\label{S0p}
 S(q,p,p)=q^p(S(q,0,p)-f_0)
\end{equation}
\begin{equation}\label{S1p}
 S(q,p+1,p)=q^p(S(q,1,p)-f_1).
\end{equation}
and, more generally,
\begin{equation}\label{Shp}
 S(q,p+h,p)=q^p(S(q,h,p)-f_h).
\end{equation}
\begin{example}
 Let $q=2$ and $p=4$. In view of (\ref{Shp1}),
\begin{align*}
R^0_\infty&\subseteq[0, S(2,0,4)]\\
R^1_\infty&\subseteq[0, S(2,1,4)]\\
R^2_\infty&\subseteq[0, S(2,0,4)+S(2,1,4)]\\
R^3_\infty&\subseteq[0, S(2,0,4)+2S(2,1,4)].
\end{align*}
See Section \ref{sqp} for the explicit calculation of $S(q,h,p)$. In Theorem \ref{thcontrollability} below, we show that above inclusions are actually equalities, so that
$$R_\infty=R^0_\infty-\frac{i}{2}R^1_\infty-\frac{1}{4}R^2_\infty+\frac{i}{8}R^3_\infty$$
 is a rectangle in the complex plane - see Figure \ref{rectangle}.
\end{example}

\begin{figure}
 \subfloat[ $R_\infty(2e^{i2\pi/3})$.]{\includegraphics[scale=0.5]{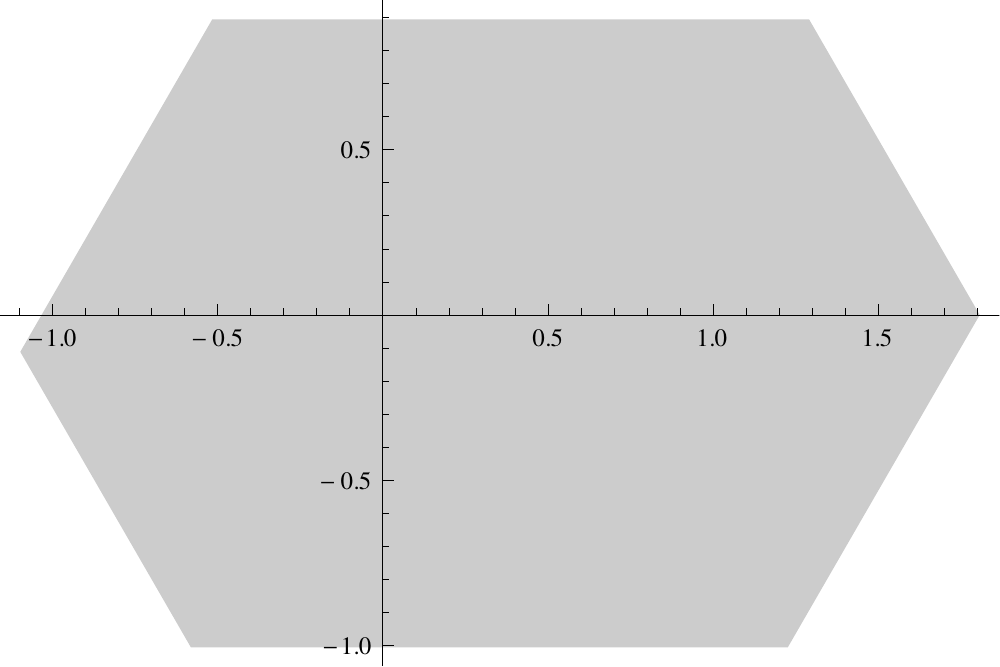}}\hskip1cm
\subfloat[ $R_\infty(2e^{i\pi/2})$.]{\includegraphics[scale=0.5]{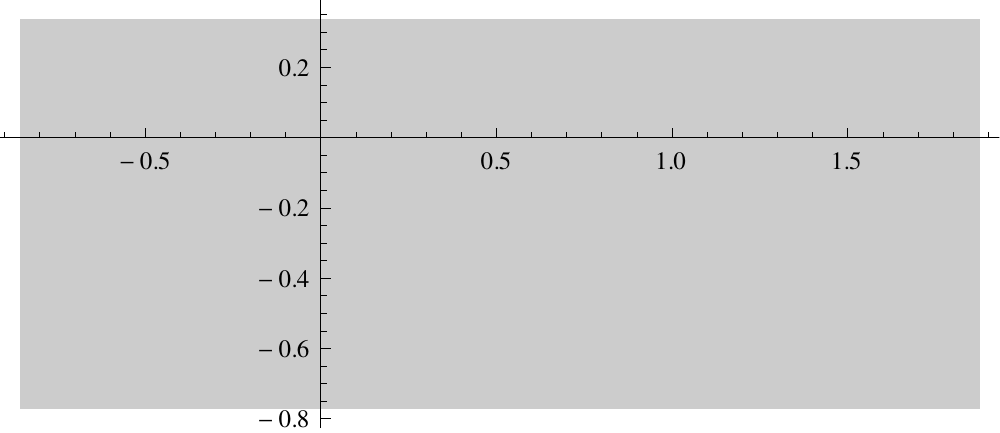}}
\caption{By Theorem \ref{thcontrollability},  $R_\infty(2e^{i2\pi/p})$ with $p=3,4$ is a polygon. \label{rectangle}}
\end{figure}

\begin{lemma}\label{lboundcomplex}
  If $q\leq q(p)$ then for every $h\in \NN$
\begin{equation}\label{estimateh}
 S(q,p,p+h)\geq q^p f_h.
\end{equation}
\end{lemma}
\begin{proof}
The case $h=0$ follows by the definition of $q(p)$ and by (\ref{S0p}). If $h=1$ then
 $$S(q,p,p+1)\geq S(q,p,p)\geq q^p f_0=q^p f_1.$$
Fix now $h\geq 2$ and now (\ref{estimateh}) as inductive hypothesis for every integer lower than $h$. It follows by (\ref{Shp1})
$$S(q,p,p+h)=f_{h-1} S(q,1,p)+f_{h-2} S(q,0,p)\geq 2 (f_{h-1}+f_{h-2})=2 f_h$$
therefore, by (\ref{Shp}), we finally get
 $$S(q,p,p+h)=q^p(S(q,h,p)-f_h)\geq q^p f_h.$$
\end{proof}

\begin{figure}\begin{center}
 \includegraphics{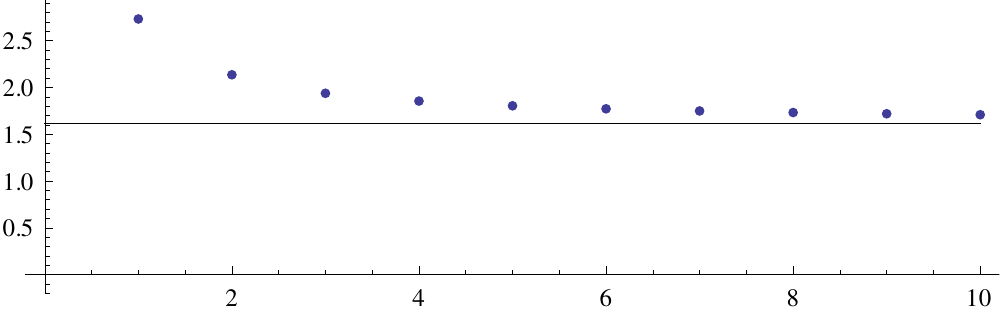}\end{center}
\caption{$q(p)$ for $p=1,\dots,10$. Note that $q(p)$ tends to $\varphi$ as $p\to\infty$. Indeed it suffices to recall $f_p\sim \varphi^p$ to have
$\lim_{p\to\infty} q(p)/\varphi=1$.}
\end{figure}
Finally let us define $q(p)$ as the greatest solution of the equation
$$S(q,0,p)=2f_0=2$$
Note that if $q\leq q(p)$ then $S(q,0,p)\geq 2$.
\begin{remark}
 The value $q(p)$ is explicitly calculated in Section \ref{sqp} below. Among other results, we shall show
\begin{equation}
 q(p)=
 \begin{cases}
 \bigg(\frac 12 (f_{p-2}+2f_p)+\frac 12 \sqrt {(f_{p-2}+2f_p)^2-8 } \bigg)^{\frac 1 p} \,\,&p\, \,\, even;\\
  \bigg(\frac 12 (f_{p-2}+2f_p)+\frac 12 \sqrt {(f_{p-2}+2f_p)^2+8 }\bigg)^{\frac 1 p}&p\,\,\,odd.
 \end{cases}
 \end{equation}
We notice that above equality implies $q(p)\sim f(p)^{1/p}\sim \varphi$ as $p\to\infty$.
\end{remark}
\begin{example}
 $q(1)=1+\sqrt{3}$, $q(2)=\sqrt{\frac{1}{2} \left(5+\sqrt{17}\right)}$, $q(3)=\sqrt[3]{\frac{1}{2} \left(7+\sqrt{57}\right)}$, $q(4)=\sqrt[4]{6+\sqrt{34}} $.
\end{example}

\begin{lemma}\label{lcontrollability}
 Let $p,h\in \NN$ and let $q\leq q(p)$. For $x\in[0,S(q,h,p)]$ consider the sequences $(r_n)$ and $(u_n)$ defined by
\begin{equation}\label{rhcomplex}
 \begin{cases}
  r_0=x;\\
  u_n=\begin{cases}
1 &\text{if } r_n\in[f_n,S(q,np+h)]\\
0 &\text{otherwise }
 \end{cases}\\
r_{n+1}=q^p(r_n-u_n f_{np+h}).
\end{cases}
\end{equation}
Then
\begin{equation}\label{xcomplex}
x=\sum_{k=0}^\infty \frac{f_{pk+h}}{q^{pk}} u_k
\end{equation}
and, consequently, $R^h_\infty=[0,S(q,0,p)]$.
Moreover if $q>q(p)$ then $R_\infty \subsetneq [0,S(q,0,p)]$.
\end{lemma}

\begin{proof}
 Fix $h\in\NN$ and $x\in[0,S(q,0,p)]$. First of all note that
\begin{equation}\label{remaindercomplex}
 x=\sum_{k=0}^n \frac{f_{pk+h}}{q^{pk}}u_k + \frac{r_{n+1}}{q^{p(n+1}} \quad \text{ for all } n.
\end{equation}
Indeed for $h=0$ one has $r_1=q^p(x-u_0f_h)$ and consequently $x=f_hu_0+r_1/q^p$. Assume now (\ref{remaindercomplex}) as inductive hypothesis. Then
\begin{align*}
r_{n+2}&=q^p(r_{n+1} - u_{n+1}f_{p(n+1)+h})\\&=q^{p(n+2)}\left( x- \sum_{k=0}^n \frac{f_{kp+h}}{q^{pk}}u_k\right) - q^{p(n+2)} f_{p(n+1)+h}u_{n+1}
\end{align*}
and, consequently,
$$x=\sum_{k=0}^{n+1} \frac{f_{kp+h}}{q^{kp}}u_k + \frac{r_{h+2}}{q^{p(n+2)}}.$$

Now, we claim that for every $n$ if $q\leq q(p)$ then
\begin{equation}\label{boundcomplex}
 r_n\in[0,S(q,pn+h,p)].
\end{equation}
We show the above inclusion by induction. If $h=0$ then the claim follows by the definition of $r_0$ and by the fact that $x\in[0,S(q,h,p)]$. Assume now (\ref{boundcomplex}) as inductive hypothesis.
One has $r_n\in[0,S(q,pn+h,p)]=[0, f_{pn+h})\cup[f_{pn+h},S(q,pn+h,p)]$. If $r_n\in [0,f_{pn+h})$ then $r_{n+1}=q^p r_n\in [0,q^p f_{pn+h}]\subseteq [0, S(q,(n+1)p+h,p)]$ - where the last inclusion follows by Lemma \ref{lboundcomplex}.
If otherwise $r_n\in [f_{pn+h},S(q,pn+h,p)]$ then $r_{n+1}=q^p(r_n-f_{np+h})\subseteq [0, q(S(q,pn+h,p)-f_{pn+h})]=[0, S(q,p(n+1)+h,p)]$ - see (\ref{Shp}).

Recalling $f_n\sim \varphi^n$ as $n\to \infty$,  one has
\begin{align*}
\sum_{k=0}^\infty \frac{f_{pk+h}}{q^{pk}}u_k&=\lim_{n\to\infty} \sum_{k=0}^{n-1} \frac{f_{pk+h}}{q^{pk}}u_k\stackrel{(\ref{remaindercomplex})}{=} x-\lim_{n\to\infty} \frac{r_n}{q^{pn}}
\stackrel{(\ref{boundcomplex})}{\geq} x-\lim_{n\to\infty}\frac{S(q,pn+h,p)}{q^{pn}}\\&\stackrel{(\ref{Shp1})}{=} x-\lim_{n\to\infty} \frac{f_{pn+h-1} S(q,1,p)+f_{pn+h-2}S(q,0,p)}{q^{ph}}=x.
\end{align*}
On the other hand
$$\sum_{k=0}^\infty \frac{f_{pk+h}}{q^{pk}}u_k=x-\lim_{n\to\infty} \frac{r_n}{q^{pn}}\leq x$$
and this proves (\ref{xcomplex}).
It follows by the arbitrariness of $x$ that if $q\leq q(p)$ then $R^h_\infty=[0,S(q,0,p)]$.
Finally assume $q>q(p)$. By Lemma \ref{lbound} there exists $x\in (S(q,h,p)/q^p,f_h)$. In order to find a contradiction, assume $x\in R^h_\infty$. Then
$$x=u_0f_h+\frac{1}{q^p}\sum_{k=0}^\infty \frac{f_{p(k+1)+h}}{q^{pk}}u_{k+1}$$
Note that $u_0\not=1$ because $x<f_h$. Then $u_0=0$ and
$$x= \frac{1}{q^p}\sum_{k=0}^\infty \frac{f_{p(k+1)+h}}{q^{pk}}u_{k+1}\leq \frac{1}{q^p}\sum_{k=0}^\infty \frac{f_{p(k+1)}}{q^{pk}}=\frac{S(q,h,p)}{q^p}$$
but this contradicts $x\in (S(q,h,p)/q^p,f_h)$. Then $x\in [0,S(q,0,p)]\setminus R^h_\infty$ and this concludes the proof.
\end{proof}
\begin{figure}\begin{center}
 \subfloat[$h=0$]{\includegraphics[scale=0.1]{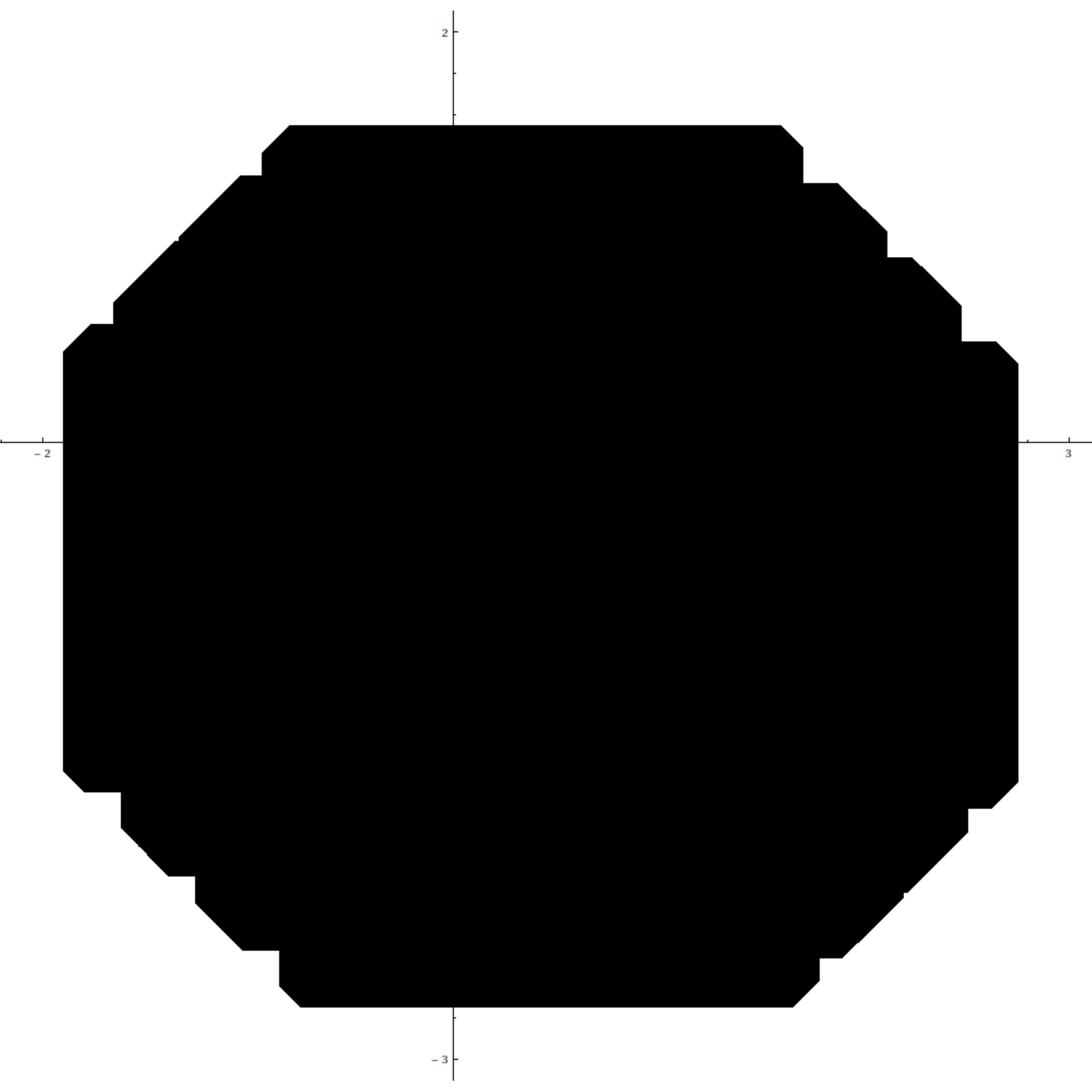}}\hskip0.1cm
\subfloat[$h=0.4$]{\includegraphics[scale=0.1]{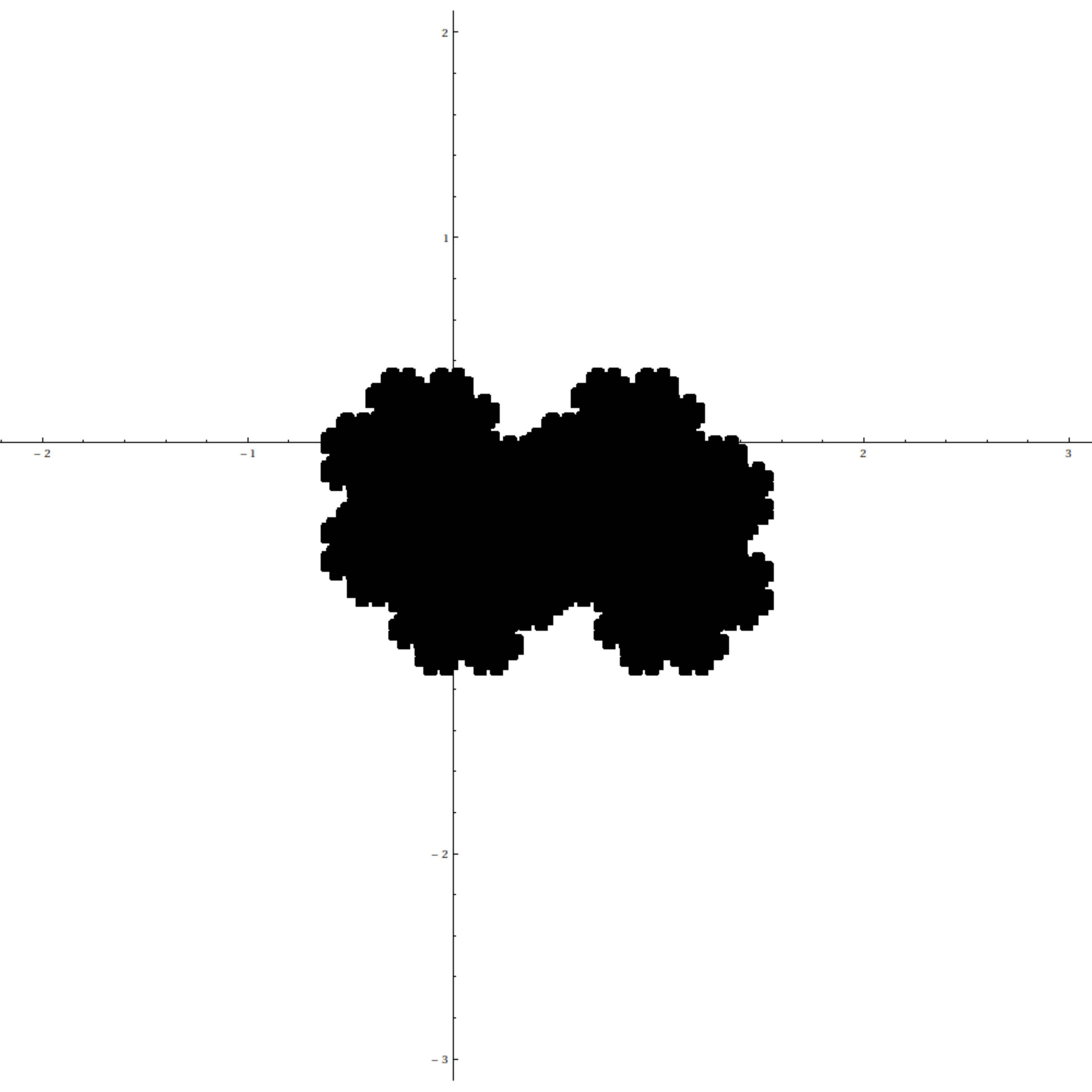}}\hskip0.1cm
\subfloat[$h=0.5$]{\includegraphics[scale=0.1]{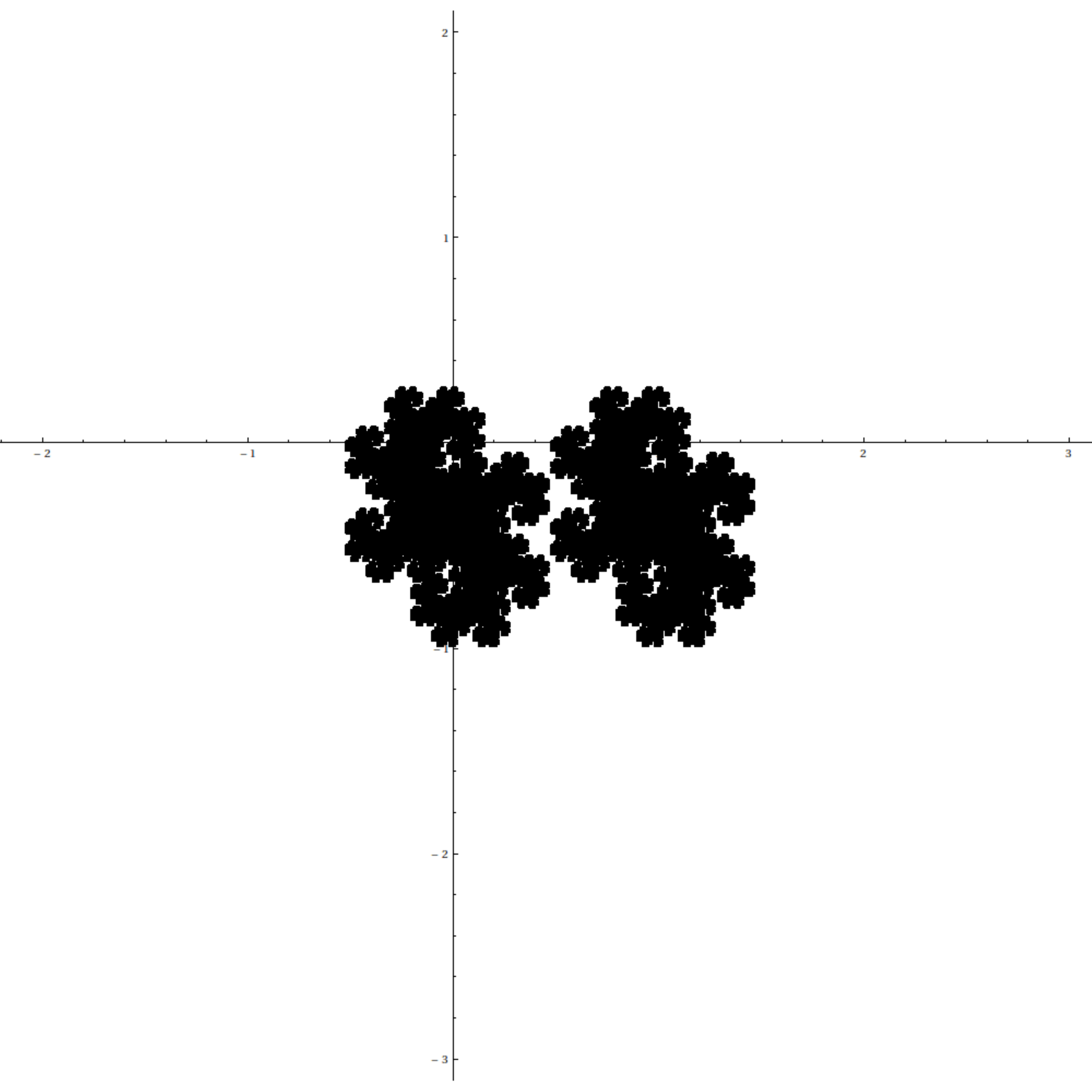}}\end{center}
\caption{Approximations of $R_\infty(z)$ with $z=(q(p)+h)e^{i\pi/4}$ and $h=0,0.4,0.5$. Note that, by Theorem \ref{thcontrollability}, if $h=0$ then $R_\infty(z)$ is indeed an octagon.
See Section \ref{sifs} and, in particular, Remark \ref{rcomplexifs} for a description of the approximation techniques.}
\end{figure}

\begin{theorem}\label{thcontrollability}
 If $\varphi<|z|\leq q(p)$ then $R_\infty(z)$ is a polygon on the complex plane containing the origin.
\end{theorem}
\begin{proof}
 It follows by Lemma \ref{lcontrollability} and by
$$R_\infty(z)=\left\{\sum_{k=0}^\infty \frac{f_k}{z^k}u_k\mid u_k\in\{0,1\}\right\}=\sum_{h=0}^{p-1} z^{-h} R_\infty^h.$$
\end{proof}
\subsubsection{Proof of Theorem \ref{workspace}}\label{proofworkspace}
Theorem \ref{workspace} immediately follows by
$$R_\infty(q^{i\omega})=\{x(\mathbf u,\mathbf 1)\mid \mathbf u\in\{0,1\}^\infty\}\subset W_{\infty,q,\omega}$$
and by Theorem \ref{thcontrollability}.

\subsubsection{Explicit calculus of $q(p)$}\label{sqp}
By a comparison between (\ref{Shp1}),(\ref{S0p}) and (\ref{S1p}), $S(q,0,p)$ and $S(q,1,p)$ are solution of the following system of equations
\begin{equation}
 \begin{cases}
  q^p(S(q,0,p)-f_0)=f_{p-1} S(q,1,p)+f_{p-2} S(q,0,p)\\
  q^p(S(q,1,p)-f_1)=f_{p} S(q,1,p)+f_{p-1} S(q,0,p)
 \end{cases}
\end{equation}
\begin{equation}
 \begin{cases}
  (q^p-f_{p-2} )S(q,0,p)-f_{p-1}S(q,1,p)= f_{0} q^p\\
  -f_{p-1} \,\,S(q,0,p)+(q^p- f_{p})S(q,1,p)= f_{1} q^p
 \end{cases}
\end{equation}
\bigskip
whose solution is
 \begin{equation}\label{S0exp}
S(q,0,p)=\frac {  \left \vert  \begin{array}{ccccc}     f_{0} q^p & -f_{p-1}\\
f_{1} q^p&q^p- f_{p}   \end{array}\right \vert}{ \left \vert  \begin{array}{ccccc}
q^p-f_{p-2} & -f_{p-1} \\
    -f_{p-1}& q^p- f_{p} \\
  \end{array}\right \vert},
 \end{equation}
 \begin{equation}\label{S1exp}
S(q,1,p)=\frac {  \left \vert  \begin{array}{ccccc}     q^p-f_{p-2}& f_{0} q^p\\
  -f_{p-1}&f_{1} q^p   \end{array}\right \vert}{ \left \vert  \begin{array}{ccccc}
q^p-f_{p-2} & -f_{p-1} \\
    -f_{p-1}& q^p- f_{p} \\
  \end{array}\right \vert}.
 \end{equation}

\noindent
We now show that the solutions in (\ref{S0exp}) and (\ref{S1exp}) are well defined.
\begin{proposition}\label{pdelta}
Let
$$
 {\Delta_p}(q):=\left \vert  \begin{array}{ccccc}
q^p-f_{p-2} & -f_{p-1} \\
    -f_{p-1}& q^p- f_{p} \\
  \end{array}\right \vert= (q^p-f_{p-2} )(q^p- f_{p})-f_{p-1}^2
$$
Then
\begin{equation}\label{delta}
\Delta_p(q)= q^{2p}-(f_{p-2} +f_{p})q^p+(-1)^{p}
 \end{equation}
 and the real roots of $\Delta_p(q)$ are $\pm \varphi$ and $\pm (\varphi-1)$ if $p$ is even and $-\varphi$ and $\varphi-1$ if $p$ is odd.\\
 In particular if $q>\varphi$ then $\Delta_p\not=0$.
\end{proposition}

\begin{proof}
The equality in (\ref{delta}) follows by Cassini identity for $p\geq 2$
$$f_{p-2}f_{p}-f_{p-1}^2=(-1)^{p}.$$
Now, we notice that $\Delta_p(q)=0$ if and only if
$$\begin{cases}
   z=q^p\\
   z^2-(f_{p-2} +f_{p})z+(-1)^p=0.
  \end{cases}$$
We first discuss the case of an even $p$. When $p$ is even then $\Delta_p(q)$ has exactly $4$ real solutions
\begin{align*}
 q^{even}_{1,2}=\pm \sqrt[p]{\frac 12 (f_{p-2} +f_{p})-\frac 12 \sqrt {(f_{p-2} +f_{p})^2+4 }},\\
 q^{even}_{3,4}=\pm \sqrt[p]{\frac 12 (f_{p-2} +f_{p})+\frac 12 \sqrt {(f_{p-2} +f_{p})^2+4 }}.
\end{align*}

Now, for every $p\in\NN$ one has that the Golden Mean $\varphi$ satisfies
$$\varphi^p=f_{p-1}\varphi+f_{p-2}$$
and, consequently,
\begin{align*}
\varphi^{2p}&=(f_{p-1}\varphi+f_{p-2})^2\\
&=f^2_{p-1}\varphi^2+2f_{p-1}f_{p-2}\varphi+f^2_{p-2}\\
&=(f^2_{p-1}+2f_{p-1}f_{p-2})\varphi+f_{p-1}^2+f_{p-2}^2.
\end{align*}
This, together with $\Delta(q)=\Delta(-q)$ and Cassini identity, implies
$$\Delta_p(\varphi)=\Delta_p(-\varphi)=f_{p-1}(f_{p-1}+f_{p-2}-f_p)\varphi+f_{p-1}^2-f_pf_{p-2}+1=0.$$
Moreover, since $\varphi-1=1/\varphi$ and $\Delta(q)=\Delta(-q)$,
$$\Delta_p(\varphi-1)=\Delta_p(1-\varphi)=\Delta_p(1/\varphi)=\frac{\Delta_p(\varphi)}{\varphi^{2p}}=0.$$
This concludes the proof for the even case. \\
Now, if $p$ is odd then $\Delta_p(q)=0$ has exactly $2$ real solutions
$$
 q^{odd}_{1,2}=\sqrt[p]{\frac 12 (f_{p-2} +f_{p})-\frac 12 \sqrt {(f_{p-2} +f_{p})^2-4 }}.
$$
Again by Cassini identity
\begin{align*}
\Delta_p(\varphi)&=\varphi^{2p}-(f_{p-2}+f_p)\varphi+1\\
&=f_{p-1}(f_{p-1}+f_{p-2}-f_p)\varphi+f_{p-1}^2-f_pf_{p-2}-(-1)^p=0.
 \end{align*}
Since $1-\varphi=-1/\varphi$ we finally obtain
$$\Delta_p(1-\varphi)=\Delta_p(-1/\varphi)=-\frac{\Delta_p(\varphi)}{\varphi^{2p}}=0.$$
\end{proof}
By Proposition \ref{pdelta}
  \begin{example}
 For $p=1$ we already showed
$$S(q,0,1)=S(q)=\frac{q^2}{q^2-q-1} \quad S(q,1,1)=S(q)=\frac{q^2+q}{q^2-q-1}$$
For $p=2$, namely when $z=-q$,
$$S(q,0,2)=\frac{q^2(q^2-1)}{q^4-3q^2+1} \quad S(q,1,2)=\frac{q^4}{q^4-3q^2+1}$$
For $p=3$, namely when $z$ is a rescaled cubic root of unity,
$$S(q,0,3)=\frac{q^3 \left(q^3-1\right)}{q^6-4 q^3-1} \quad S(q,1,3)=\frac{q^6+q^3}{q^6-4 q^3-1}$$
For $p=4$
$$S(q,0,4)=\frac{q^4(q^4-2)}{q^8-7q^4+1} \quad S(q,1,4)=\frac{q^8+q^4}{q^8-7q^4+1}$$
\end{example}
We now give a closed formula for $q(p)$.
\begin{proposition}\label{proppq}
For every $p\in\NN$
\begin{equation}\label{q(p)}
 q(p)=
 \begin{cases}
 \bigg(\frac 12 (f_{p-2}+2f_p)+\frac 12 \sqrt {(f_{p-2}+2f_p)^2-8 } \bigg)^{\frac 1 p} \,\,&p\, \,\, even;\\
  \bigg(\frac 12 (f_{p-2}+2f_p)+\frac 12 \sqrt {(f_{p-2}+2f_p)^2+8 }\bigg)^{\frac 1 p}&p\,\,\,odd.
 \end{cases}
 \end{equation}
 \end{proposition}
\begin{proof}
 We recall that $q(p)$ is defined as the greatest solution of
$\sum_{k=0}^\infty \frac{f_{kp}}{q^{kp}}=2$
namely of
 \begin{equation*}
S(q,0,p)=\frac{  q^{2p}- f_{p-2} q^p}
 {q^{2p}-(f_{p-2} +f_{p})q^p+(-1)^p }=2.
 \end{equation*}
Solving above equation one gets
 \begin{equation*}
  {q^{2p}+(-f_{p-2} -2f_p)q^p+2(-1)^p }=0 \end{equation*}
 and finally (\ref{q(p)}).
\end{proof}

\section{A characterization of the reachable set via Iterated Function Systems}\label{sifs}
\subsection{Some basic facts about IFSs}
An iterated function system (IFS) is a set of contractive functions $G_j:X\to X$, where $(X,\mathbf d)$ is a metric space.
 We recall that a function if for every $x,y\in X$
$$d(f(x),f(y))< c\cdot d(x,y)$$
for some $c<1$.
In \cite{Hut81} Hutchinson showed that every finite IFS, namely every IFS with finitely many contractions,
admits a unique non-empty compact fixed point $Q$ with respect to the Hutchinson operator $$\mathcal G: S\mapsto \bigcup_{j=1}^J G_j(S).$$
 Moreover for every non-empty compact set $S\subseteq \mathbb C$
$$\lim_{k\to\infty} \mathcal G^k(S)=Q.$$
The \emph{attractor} $Q$ is a self-similar set and it is the only bounded set satisfying  $\mathcal F(Q)=Q$.

\subsection{The reachable set is a projection of the attractor of an IFS}
Let $q>\varphi$, $\mathbf v\in \RR^2$ and consider the linear map from $\RR^2$ onto itself
$$ F_{q,\mathbf v}(\mathbf x)= \mathbf v+ A(q) \mathbf x$$
where
$$A(q)= \begin{pmatrix}
         \frac{1}{q} & \frac{1}{q^2}\\
         1&0
        \end{pmatrix}$$
One has
$$\begin{pmatrix}
   x_{n+2}\\
   x_{n+1}
  \end{pmatrix}= \begin{pmatrix}
u_{n+2}\\
0
\end{pmatrix}+\begin{pmatrix} \frac{1}{q} & \frac{1}{q^2}\\
         1&0\end{pmatrix}\begin{pmatrix}
x_{n+1}\\
x_n
\end{pmatrix}$$
namely
\begin{equation}\label{f}
(x_{n+2},x_{n+1})^T= F_{q,(u_{n+2},0)} (x_{n+1},x_n)^T.
\end{equation}
We now introduce the concept at the base of the symbolic dynamics, which is a particular application from $\mathbf u\in \{0,1\}^\infty$ into itself that iterates in a natural way.
\begin{definition}
The application $\sigma:\{0,1\}^\infty\rightarrow \{0,1\}^\infty$ defined by
\begin{equation}
\sigma(\mathbf u)=\sigma(u_0,u_1,u_2,...)=(u_1,u_2,...)
\end{equation}
it is said \emph{unit shift}.
\end{definition}
 Set
\begin{align*}
Q_\infty&:=\{(x(\mathbf u),x(\sigma (\mathbf u))\mid \mathbf u\in\{0,1\}^\infty\}\\
&=\left\{\left(\sum_{k=0}^\infty \frac{f_k}{q^k}u_k,\sum_{k=0}^\infty \frac{f_k}{q^k}u_{k+1}\right)\mid \mathbf u\in\{0,1\}^\infty\right\}.
 \end{align*}
\begin{proposition}\label{feq}
For every $q>\varphi$
$$\bigcup_{u\in\{0,1\}}F_{q,(u,0)}(Q_\infty)=Q_\infty.$$
\end{proposition}
\begin{proof}
 Let $\mathbf u=(u_0,u_1,\dots)\in\{0,1\}$. One has
$$F_{q,(u_0,0)}(x(\sigma(\mathbf u)),x(\sigma^2(\mathbf u)))=(x(\mathbf u),x(\sigma(\mathbf u)))$$
and this implies $Q_\infty\subseteq \bigcup_{u\in\{0,1\}} F_{q,(u,0)}(Q_\infty).$
Now let $u\in\{0,1\}$ and $\mathbf d\in\{0,1\}^\infty$. Define $\mathbf u=(c,\mathbf d)=(c,d_0,d_1,\dots)$ and note that $\sigma(\mathbf u)=\mathbf d$. One has
$$F_{q,(u,0)}(x(\mathbf d),x(\sigma(\mathbf d))=(x(\mathbf u),x(\mathbf d))=(x(\mathbf u),x(\sigma(\mathbf u)))$$
and this implies the inclusion $\bigcup_{u\in\{0,1\}} F_{q,(u,0)}(Q_\infty)\subseteq Q_\infty.$
\end{proof}

Note that in general $F_{q,\mathbf v}$ is not a contractive map. However the spectral radius of $A(q)$, say $\rho(q)$, satisfies
$$\rho(q)=\frac{\varphi}{q}<1 \quad \text{for every } q >\varphi$$
Then
$$\lim_{k\to\infty} A^k(q)=0$$
In particular there exists $k(q)$ such that for every $k\geq k(q)$
$$|| A^k(q)||:= \max_{\mathbf x\not=(0,0)} \frac{||A^k(q)\mathbf x||}{||\mathbf x||} <1.$$

\begin{example}
Let $k=2$. One has
$$
||A^2(q)||_2=\frac{q^4+5q^2+1}{q^6}
$$
- see Section \ref{kq} below for a detailed computation of $||A^k(q)||$.
Therefore $||A^2(q)||< 1$ if and only if
$$q^6-q^4-5q^2-1>0,$$
namely $k(q)=2$ for every $q> \bar q\simeq 1.69299$ where $\bar q$ is the unique positive solution of equation $q^6-q^4-5q^2-1=0$.
\end{example}

Now, for every binary sequence of length $k$, say $\mathbf u_k$, define the vector
$$\mathbf v(\mathbf u_k):=\sum_{h=0}^{k-1} A^h(q)\begin{pmatrix}
                                               u_{k+1-h}\\0\end{pmatrix}.$$
Then for every $k$ let
$$G_{q,\mathbf u_k}(\mathbf x)=\mathbf v(\mathbf u_k)+A^k(q)\mathbf x=\sum_{h=0}^{k-1} A^h(q)\begin{pmatrix}
                                               u_{k+1-h}\\
					      0\end{pmatrix}
					      +A^k(q)\mathbf x$$
One has that for $k=1$
\begin{equation}
G_{q,\mathbf u_1}=F_{q,(u_2,0)}
\end{equation}
and, more generally,
\begin{equation}\label{g}
G_{q,\mathbf u_k}=F_{q,(u_{k+1},0)}\circ F_{q,(u_{k},0)}\circ \cdots \circ  F_{q,(u_{2},0)}.
\end{equation}
\begin{remark}
If $\mathbf u_k=(u_{n+2},\cdots, u_{n+1+k})$ then
$$(x_{n+1+k},x_{n+k})^T= G_{q,\mathbf u_k} (x_{n+1},x_n)^T.$$
 \end{remark}

\begin{theorem}\label{pifs}
 For $k\geq k(q)$ and for every $\mathbf u_k\in\{0,1\}^k$ the map $G_{q,\mathbf u_k}$ is a contraction and
\begin{equation}\label{ifseq}
\bigcup_{\mathbf u_k\in \{0,1\}^k} G_{q,\mathbf u_k}(Q_\infty)=Q_\infty.
\end{equation}
 Moreover $Q_\infty(q)$ is the attractor of a two-dimensional linear Iterated Function System (IFS)
$$\mathcal G_{q,k}:=\{G_{q,\mathbf u_k}\mid \mathbf u_k\in\{0,1\}^k\},$$
namely for every compact set $X\subset \RR^2$ one has
$$\lim_{n\to\infty} \mathcal G^n_{q,k}(X)=Q_\infty(q).$$
\end{theorem}
\begin{proof}
 By the definition of $k(q)$, $G$ is a contractive map. The equality (\ref{ifseq}) follows by Proposition \ref{feq} and by (\ref{g}). The second part of the statement follows
by the fact that in general the attractor of an IFS is its unique invariant compact set, see for instance \cite{Fal90}.
\end{proof}
\begin{figure}\begin{center}
 \subfloat[$q=2$]{\begin{tabular}{c}\includegraphics[scale=0.5]{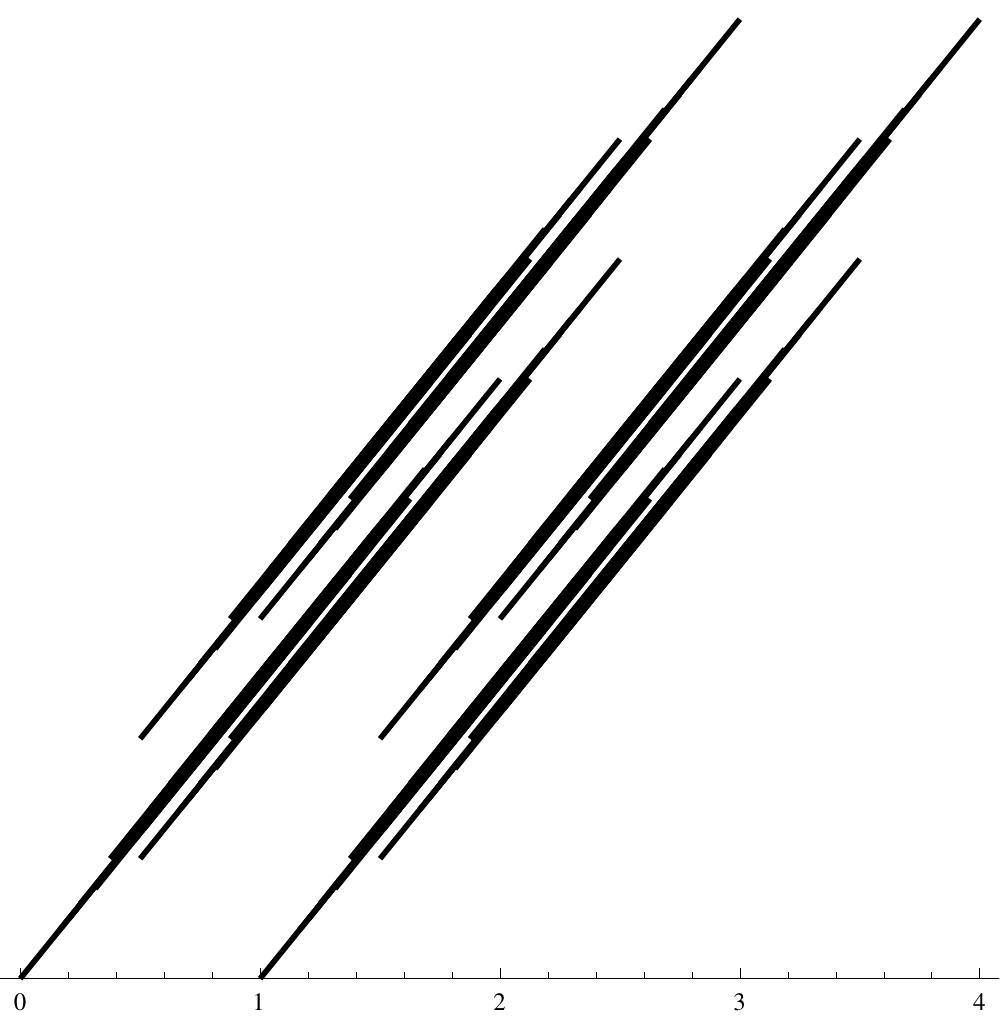}\\\includegraphics[scale=0.5]{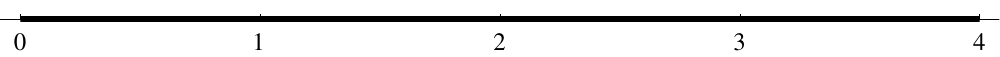}\end{tabular}}\hskip1cm
 \subfloat[$q=3$]{\begin{tabular}{c}\includegraphics[scale=0.5]{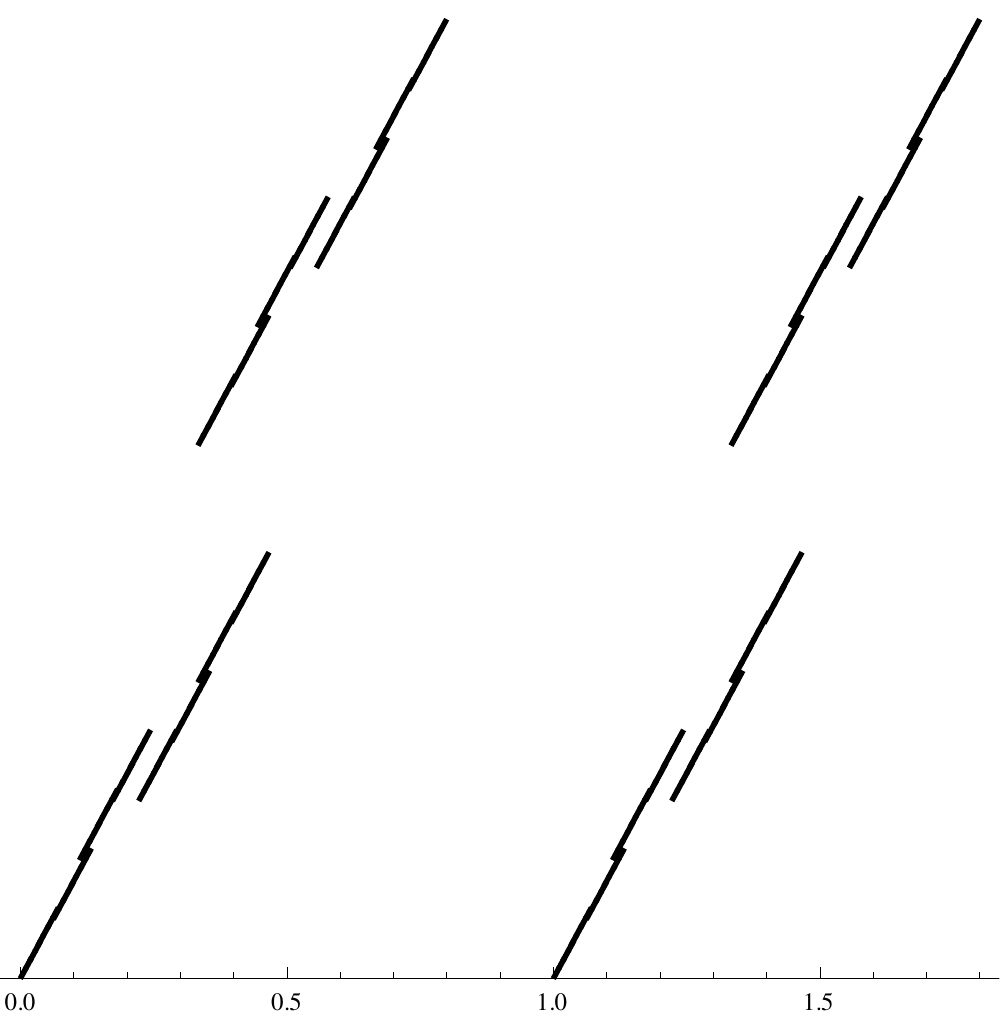}\\\includegraphics[scale=0.5]{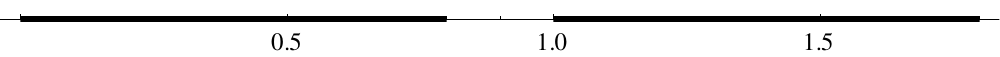}\end{tabular}}
                            \end{center}
\caption{\label{qinfty} An approximation of $Q_\infty(q)$, with $q=2,3$, and of its projection on $x$-axis $R_\infty(q)$. It is obtained by $4$ iterations of the IFS
$\mathcal G_{q,2}$ with initial datum $[0,S(q)]\times[0,S(q)]$.}
\end{figure}

 \begin{remark}[Some remarks on the approximation of $R_\infty$ in the complex case.]\label{rcomplexifs}
  Theorem \ref{pifs} gives an operative way to approximate $Q_\infty(q)$ and, consequently, $R_\infty(q)$, see Figure \ref{qinfty}.\\
Above reasonings clearly apply when considering as a base a complex number $z=q e^{i\omega}$, so that $Q_\infty(z)\subset \mathbb C\times \mathbb C$. Note that
$$Q_\infty (z)\subset H(z):=\{(z_1,z_2)\in \CC\times \CC\mid \max\{|\Re(z_{h})|,|\Im(z_{h})|\}\leq S(|z|),~h=1,2\}$$
 and $\lim_{n\to \infty} \mathcal G^n_{z,k}( H(z))=Q_\infty (z)$.
Then one may approximate $Q_\infty (z)$ by iteratively applying $\mathcal G_{z,k}$ to $H(z)$. To this end, it is possible to employ the isometry between $\CC$ and $\RR^2$ in order to set the problem on $\RR^4$.
Then the real-valued counterpart of $H(z)$ is the hypercube
$$\tilde H(z):=\{\mathbf x \in \RR^4 \mid |\mathbf x|_{\max} \leq S(|z|)\}$$
while we denote by $\tilde {\mathcal G}_{z,k}$ and by $\tilde{G}_{z,\mathbf u}$ the real-valued counterparts of ${\mathcal G}_{z,k}$ and of ${G}_{z,\mathbf u}$, respectively, so that
$$\mathcal G^n_{z,k}(x)=\bigcup_{\mathbf u\in \{0,1\}^{nk}} G_{z,\mathbf u}(x).$$
 We then may get a bidimensional representation of an approximation of $R_\infty(z)$ by projecting $\tilde G^n_{z,k}(\tilde H(z))$ on $\RR^2$.
 However this yields some complexity issues in numerical simulations.
Indeed a brute force attack consists in applying $\tilde{\mathcal G}^n_{z,k}$ to a four-dimensional grid rastering
$\tilde H(z)$ and then projecting the result on $\RR^2$. Thus the generation of an image with5 $N\times N$ pixels involves the computation of $2^{kn} N^4$ points.

In order to restrain the computational cost, we employed the geometric properties of $\tilde G_{q,\mathbf u_k}$.
Indeed for every $\mathbf u$, $\tilde G_{z,\mathbf u}$ is an affine map, thus it preserves parallelism and convexity.
 In view of these properties we considered only the $16$ vertices of $\tilde H(z)$, say $\mathbf x_j$, with $j=1,\dots,16$. Our method consists in computing the
 $\tilde G_{z,\mathbf u}(\mathbf x_j)$'s separately, in projecting the result (namely $2^{kn}$ points) on $\RR^2$ and finally on computing their convex hull, employing the fact that this
projection, say $\pi$, preserves convexity, too.  In other words we employed the identity
$$\pi(\tilde {G}_{z,\mathbf u}(\tilde H(z)))=\pi(\tilde {G}_{z,\mathbf u}(co(\{\mathbf x_j\})))=co(\pi(\tilde G_{z,\mathbf u}(\mathbf x_j))),$$
so that
$$\tilde G^n_{z,k}(\tilde H(z))=\bigcup_{\mathbf u\in\{0,1\}^{kn}}co(\pi(\tilde G_{z,\mathbf u}(\mathbf x_j))).$$
With this method we need to compute $2^{kn}\cdot 16$ points and we may possibly store the result on a vectorial format, instead of a raster one.
See Figure \ref{rinftyz2} and Figure \ref{rinftyz} for some examples.
\end{remark}
\begin{figure}\begin{center}
\subfloat[$n=1$]{\includegraphics[scale=0.11]{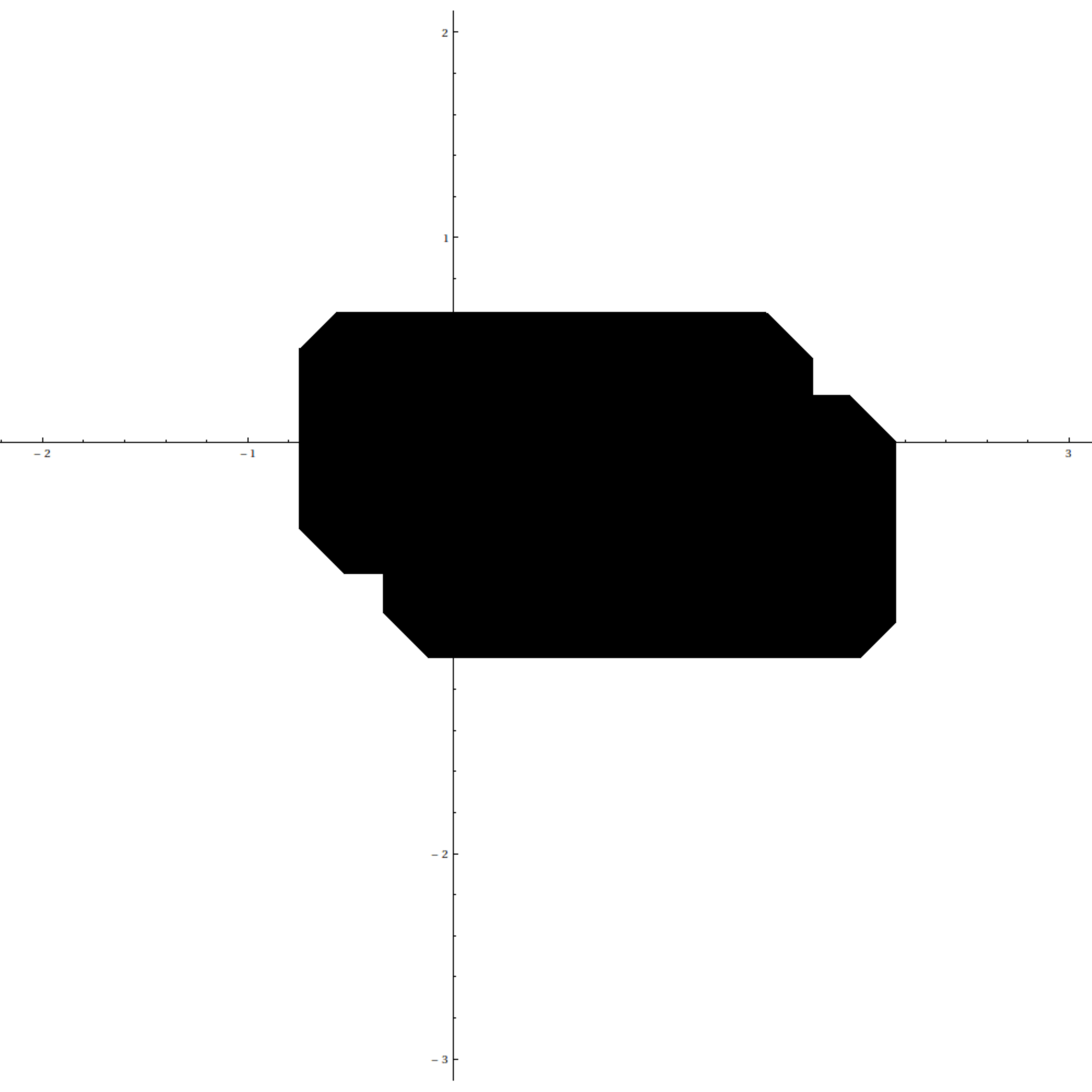}}\hskip0.1cm
\subfloat[$n=2$]{\includegraphics[scale=0.11]{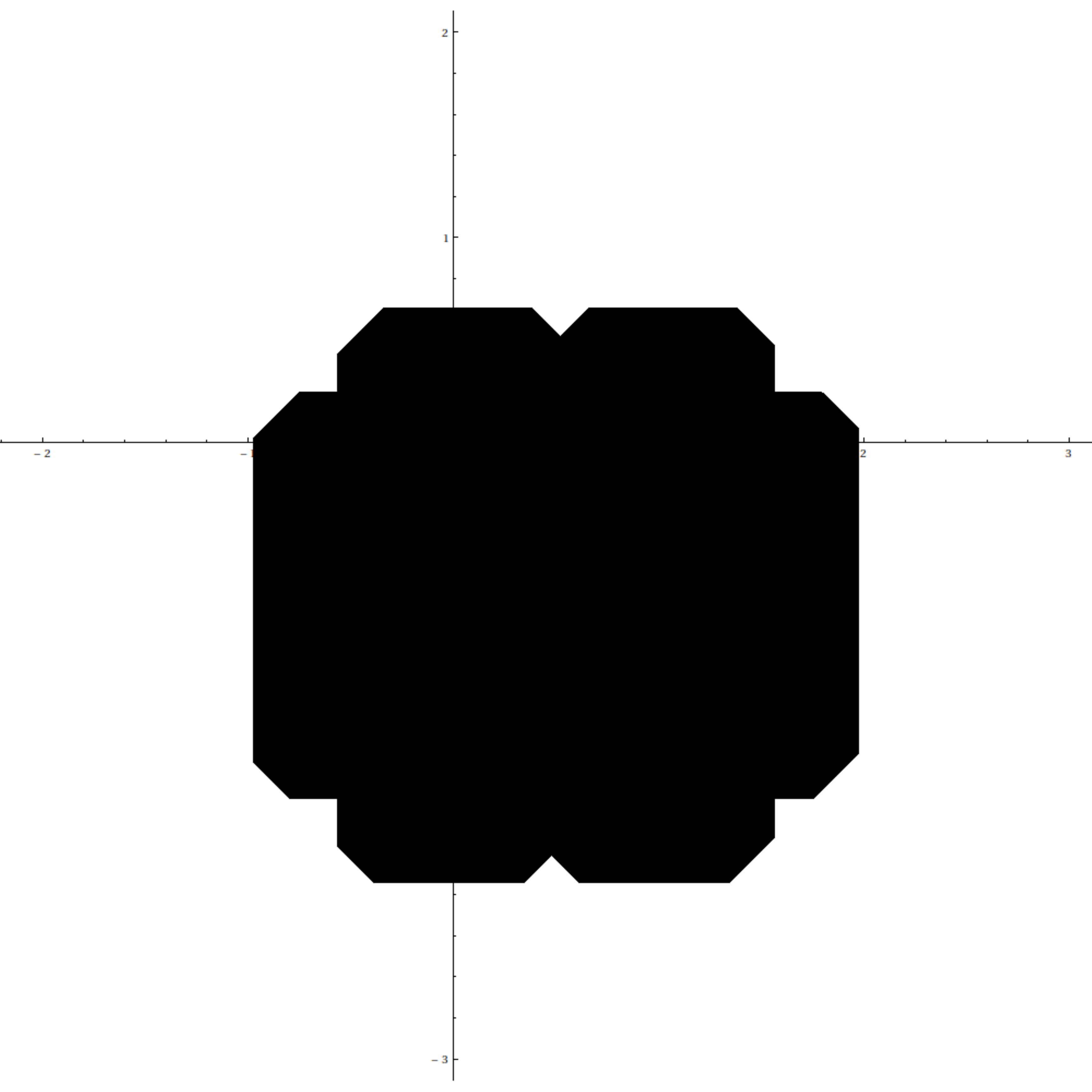}}\hskip0.1cm
\subfloat[$n=3$]{\includegraphics[scale=0.11]{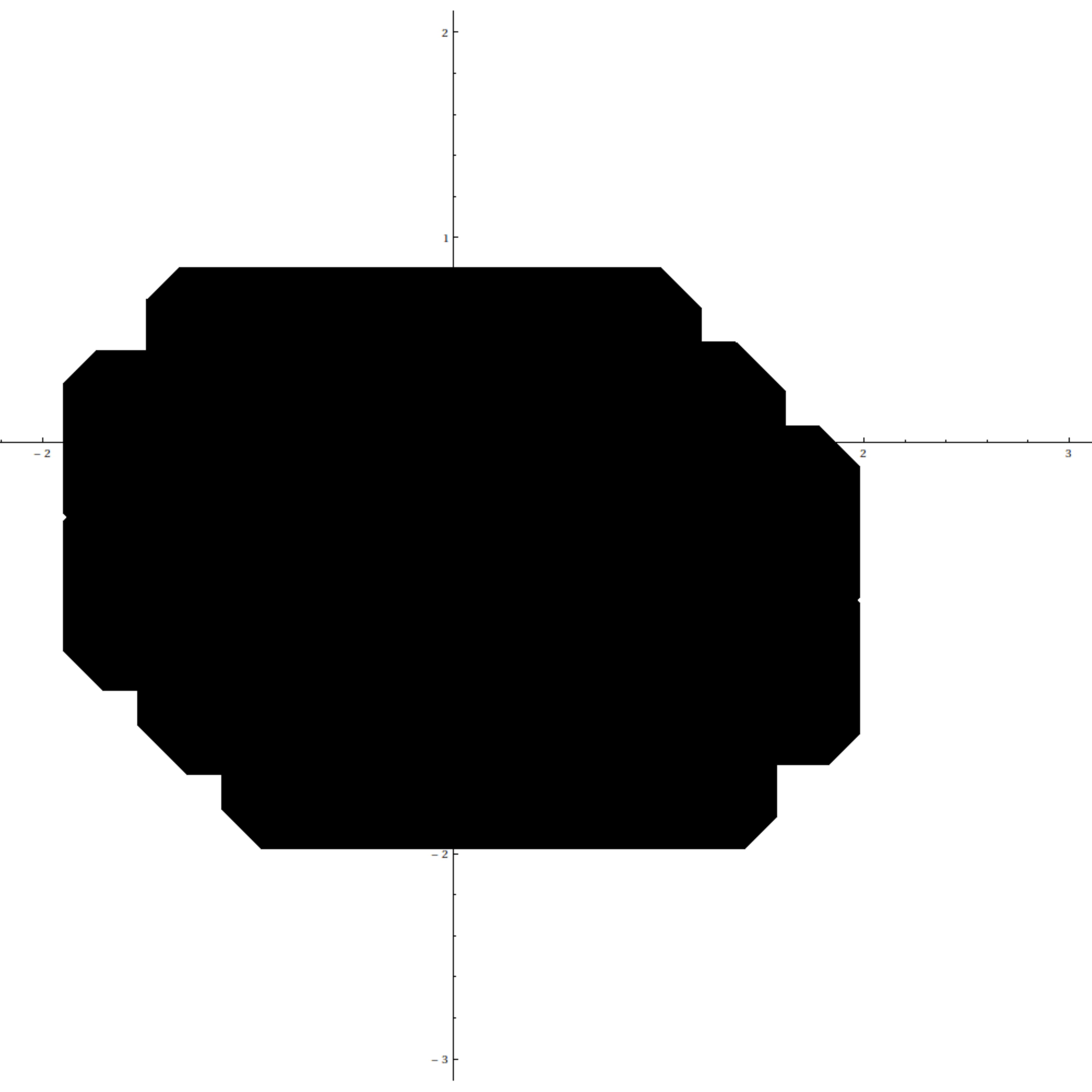}}\\
\subfloat[$n=4$]{\includegraphics[scale=0.11]{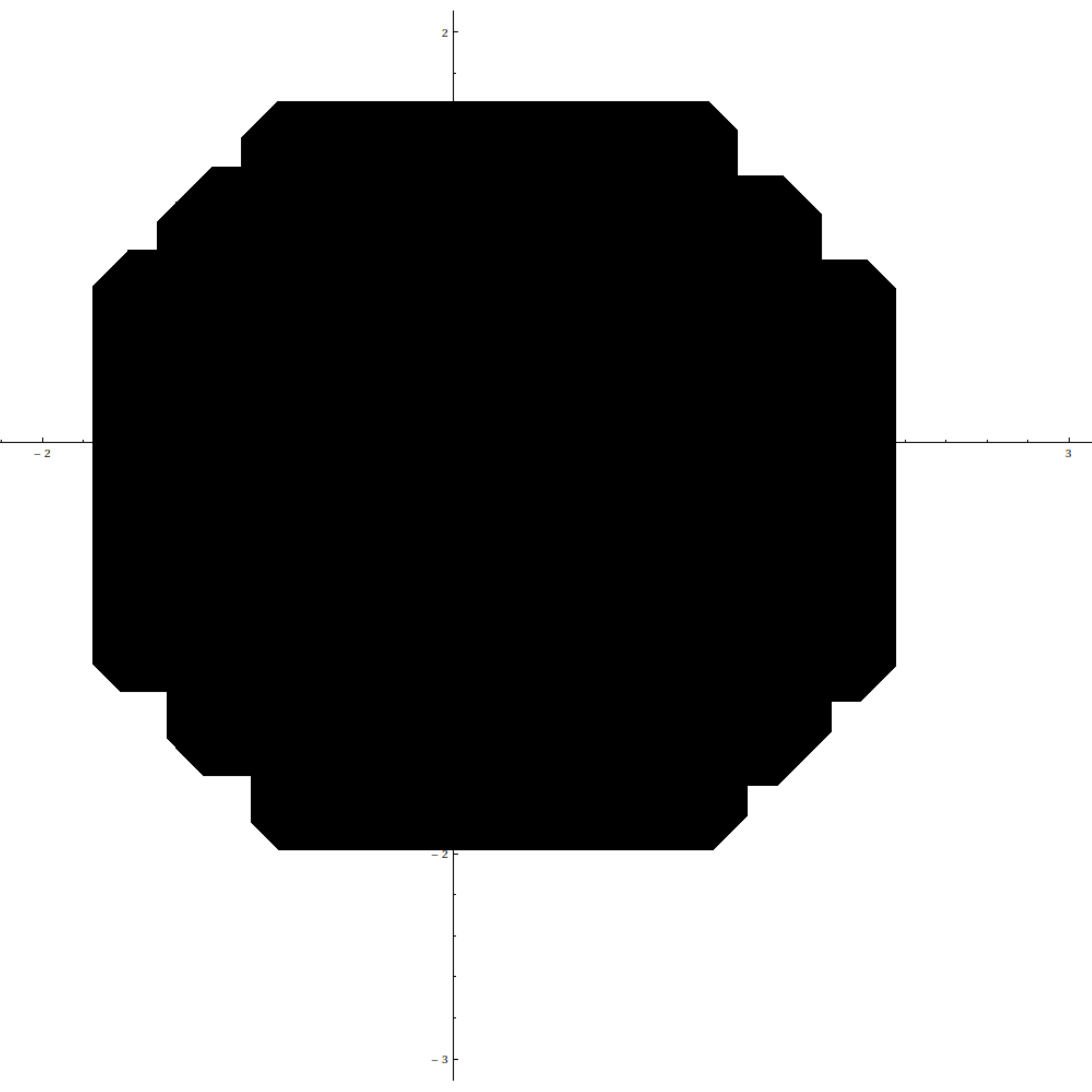}}\hskip0.1cm
\subfloat[$n=5$]{\includegraphics[scale=0.11]{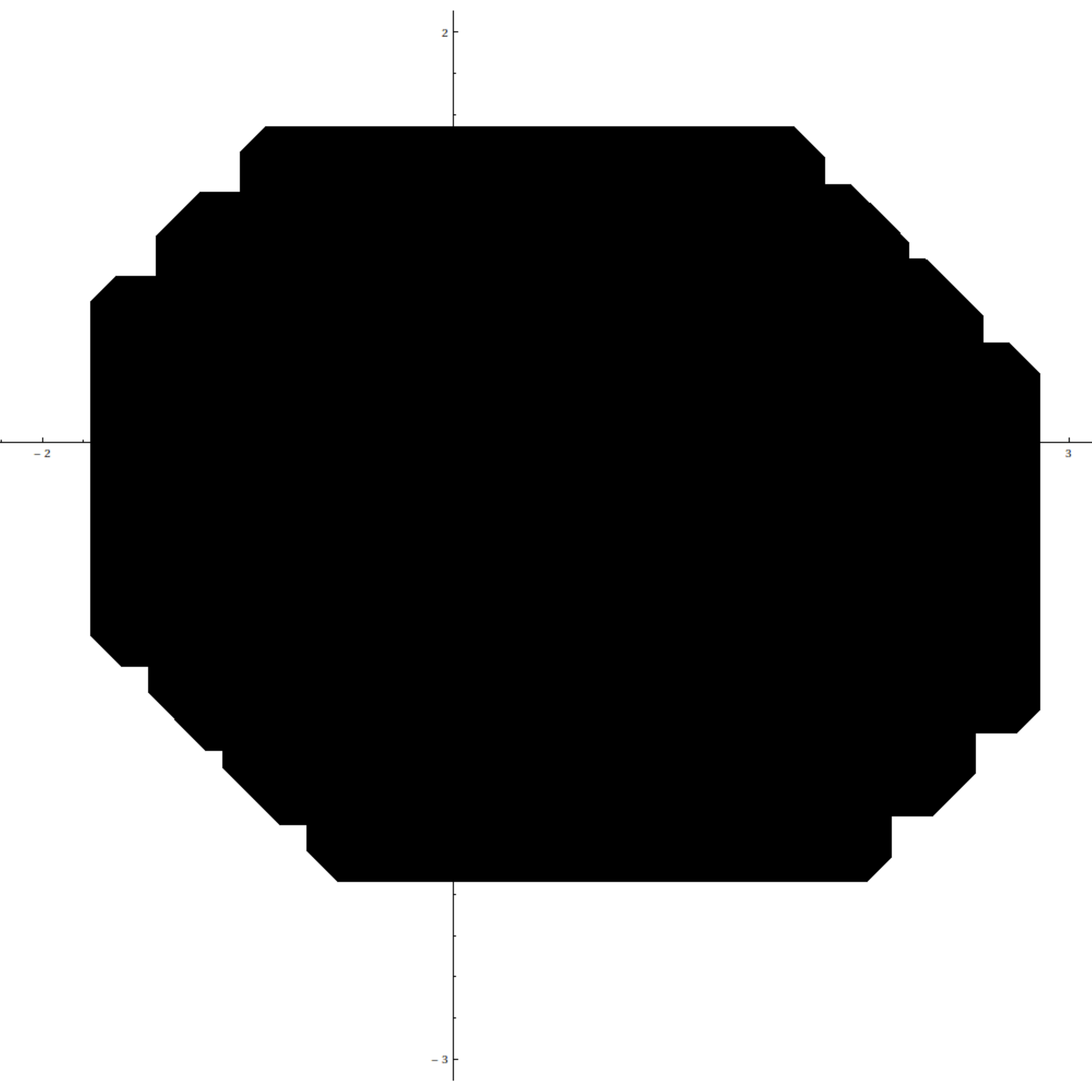}}\hskip0.1cm
\subfloat[$n=6$]{\includegraphics[scale=0.11]{test_q0_n12.pdf}}
                            \end{center}
\caption{\label{rinftyz} Various iterations of $\tilde{\mathcal G}^n_{z,k} (\tilde H(z))$ with $z=q(8)e^{i\pi/4}$. }
\end{figure}

\begin{figure}
\begin{center}
\subfloat[$n=1$]{\includegraphics[scale=0.11]{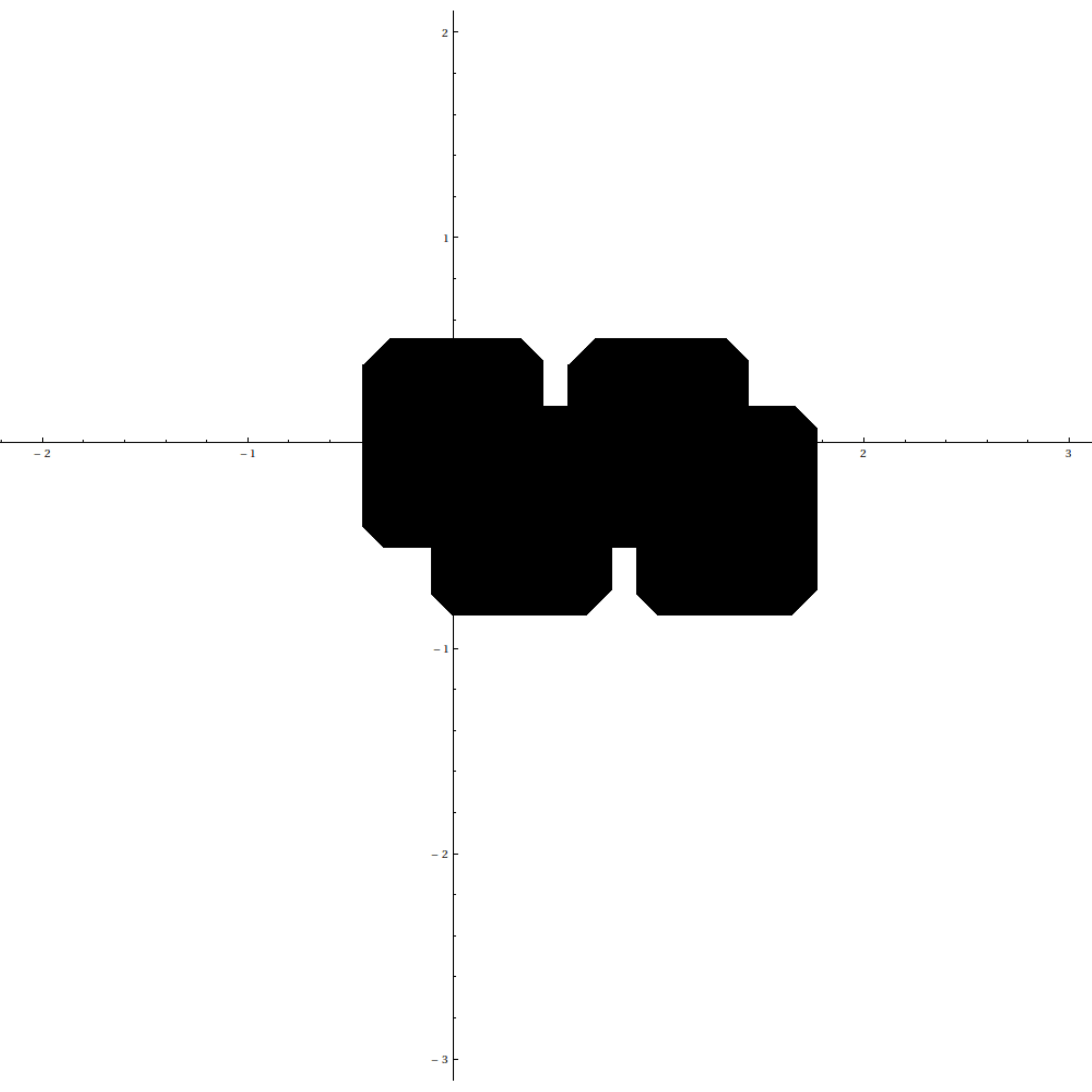}}\hskip0.1cm
\subfloat[$n=2$]{\includegraphics[scale=0.11]{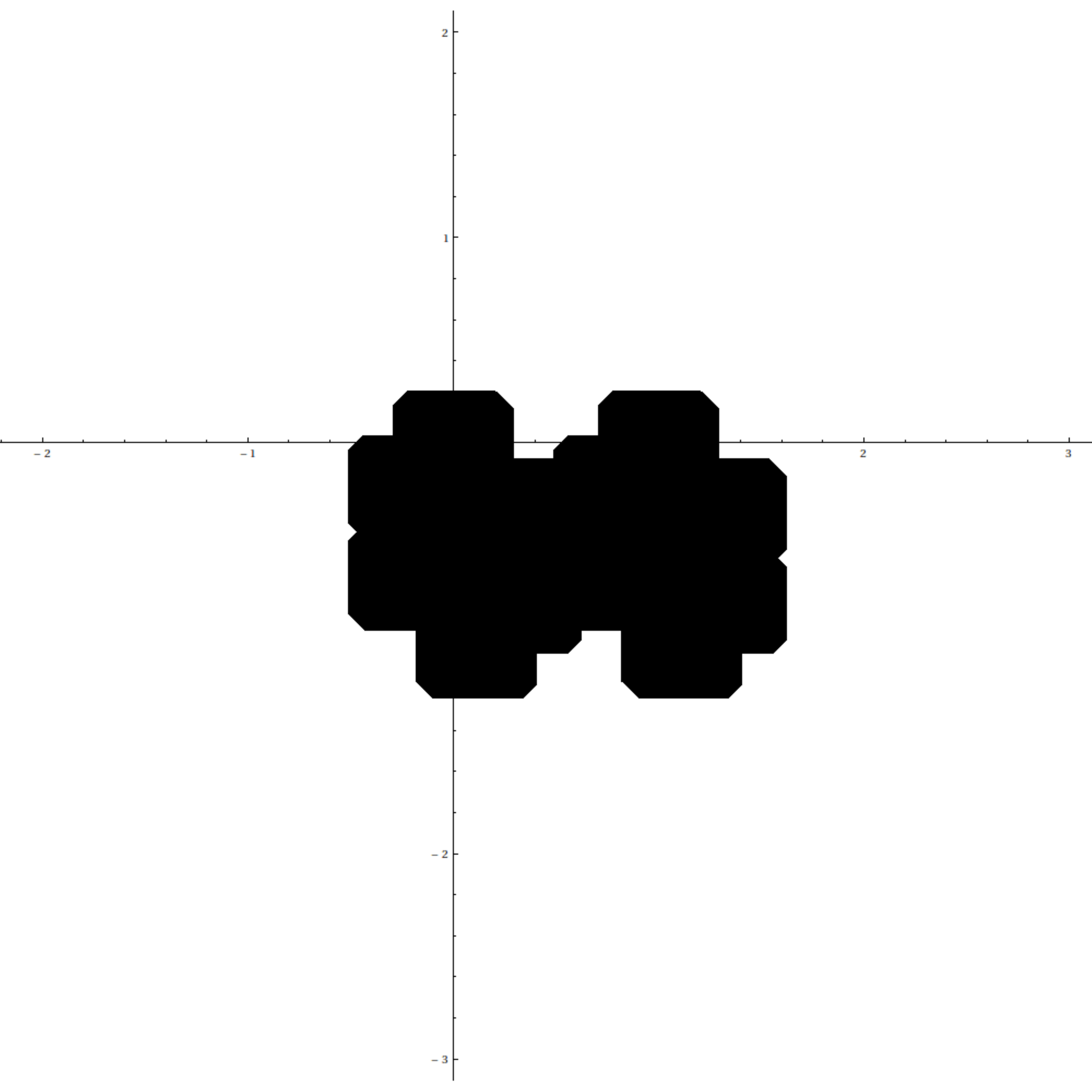}}\hskip0.1cm
\subfloat[$n=3$]{\includegraphics[scale=0.11]{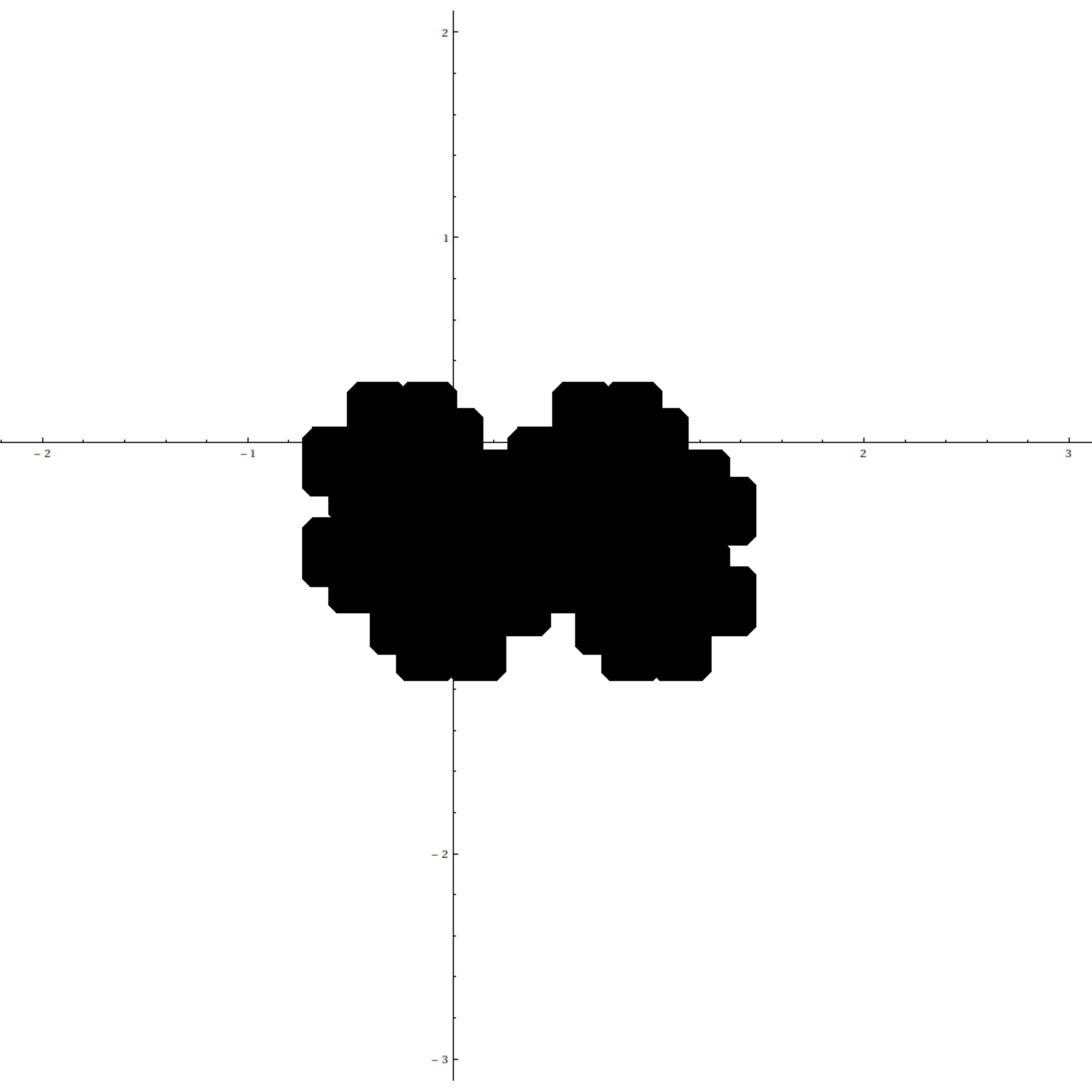}}\\
\subfloat[$n=4$]{\includegraphics[scale=0.11]{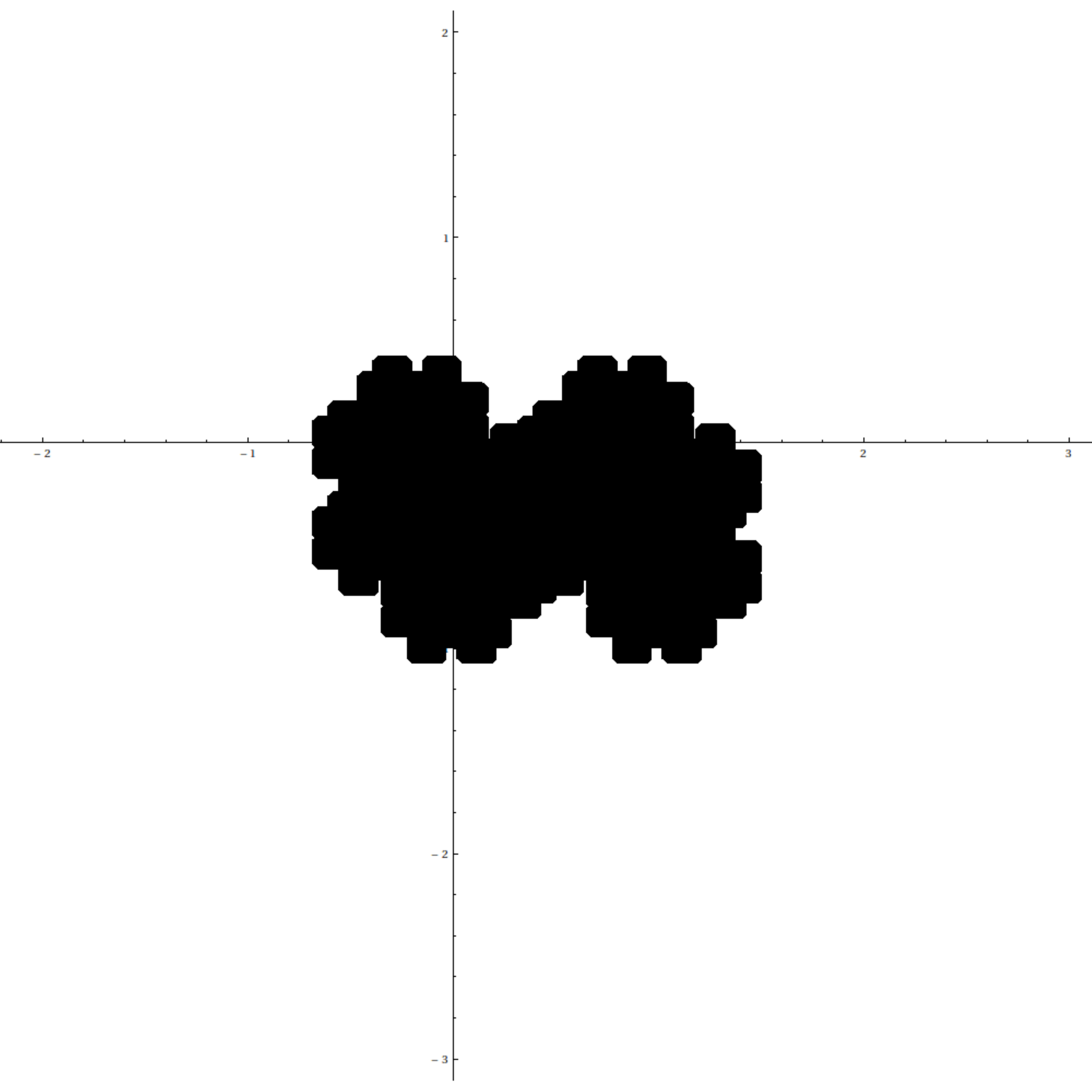}}\hskip0.1cm
\subfloat[$n=5$]{\includegraphics[scale=0.11]{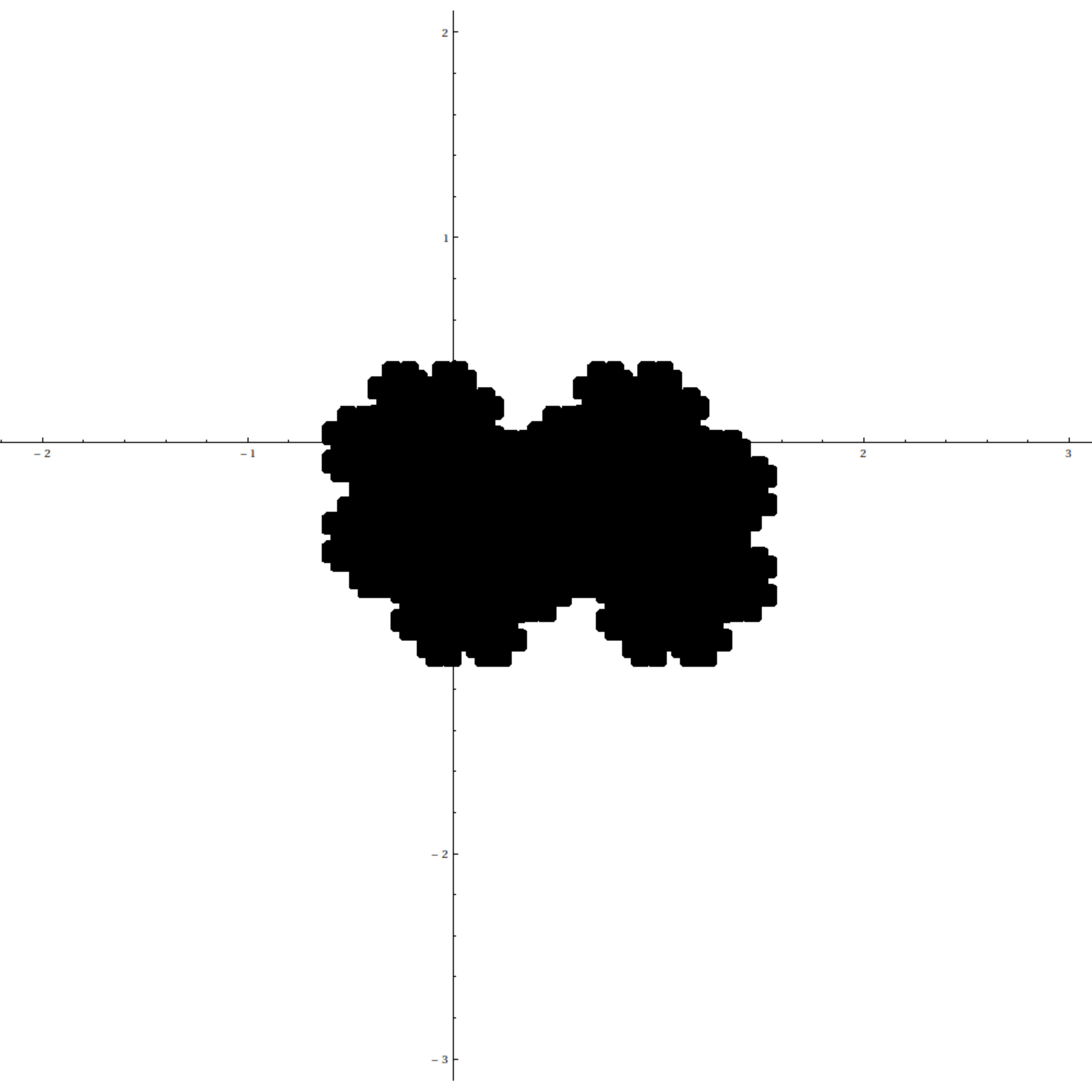}}\hskip0.1cm
\subfloat[$n=6$]{\includegraphics[scale=0.11]{test_q4_n12.pdf}}
 \end{center}
\caption{\label{rinftyz2} Various iterations of $\tilde{\mathcal G}^n_{z,2} (\tilde H(z))$ with $z=(q(8)+0.3)e^{i\pi/4}$. Notice the similarity with the twin-dragon curve, generated by
expansions in complex base with argument again $\pi/4$. }
\end{figure}

\begin{remark}[Some remarks on the analogies with expansions in non-integer bases]
We notice that the $G_{q,\mathbf u_k}$'s share the same scaling factor, $A^k(q)$, and they differ for the translation component $\mathbf v(\mathbf u_k)$. A similar structure also emerges for the
one-step recursion case, generating power series with coefficients in $\{0,1\}$. Indeed
\begin{equation}\label{1dc}
x_n=\sum_{k=0}^n\frac{u_{n-k}}{q^k}\quad \Leftrightarrow \quad \begin{cases}
                                                              x_0=u_0\\
                                                              x_{n+1}=u_{n+1}+\frac{x_n}{q}.
                                                             \end{cases}
\end{equation}
and setting
$$\bar R_\infty:=\left\{\sum_{k=0}^\infty\frac{u_k}{q^k}\mid u_k\in\{0,1\}\right\}$$
one has that
$$\bar R_\infty = \bigcup_{u\in\{0,1\}} \bar G_{q,c}(\bar R_\infty)$$
where $$\bar G_{q,c}(x)=c+\frac{x}{q}.$$
The differences and analogies between the two systems can be summarized as follows
\begin{enumerate}
 \item both systems converge to power series;
\item $\bar R_\infty$ can be generated by a one-step recursive algorithm and it is the attractor of a one-dimensional IFS,  the radius of convergence is  $1$. The buffer needed (i.e. the number of digits the IFS depends on) is constantly equal to
$1$;
\item $R_\infty$ can be generated by a two-steps recursive algorithm and it is the attractor of a two-dimensional IFS, the radius of convergence is  $\varphi$. The buffer needed, $k(q)$,  depends on $q$ and it goes to infinity as $q$ tends to $\varphi$ from above.
\end{enumerate}
 \end{remark}

\subsection{A sufficient contractivity condition}\label{kq}
In what follows we provide an upper estimate for $k(q)$.
\begin{proposition}[An upper estimate for the Fibonacci sequence]
\label{Fnphi}
For every $n\in\NN$
$$
f_{n+1}\leq \varphi^n.
$$
\end{proposition}
\begin{proof}
By induction on $n$. First, as base cases, we will consider the cases when $n = 1$ and $n = 2$. Note that $1 < \varphi < 2$. By adding 1 to each term in the inequality, we obtain
 $2 < \varphi + 1 < 3$. The two inequalities together yield
$$1 < \varphi < 2 < \varphi + 1 < 3.$$
Using the relation $\varphi + 1=\varphi^2$ and the first few Fibonacci numbers, we can rewrite this as
$$f_2 < \varphi < f_3 < \varphi^2 < f_4$$
which shows that the statement is true for $n = 1$ and $n = 2$. Now, as the induction hypothesis, suppose that $f_{i+1} < \varphi_i < f_{i+2}$ for all $i$ such that $0 \leq i \leq k + 1$.
$$f_{k+2} < \varphi^{k+1} < f_{k+3}$$
and
$$f_{k+1} < \varphi^{k} < f_{k+2}.$$
Adding each term of the two inequalities, we obtain
$$f_{k+2} + f_{k+1} < \varphi^{k+1} + \varphi^{k} < f_{k+3} + f_{k+2}.$$
Using the relation $\varphi^{k+1} + \varphi^{k}=\varphi^{k+2}$ and the first few Fibonacci numbers, we can rewrite this inequality as
$$f_{k+3} < \varphi^{k+2} < f_{k+4}$$
which shows that the inequality holds for $n = k+ 2$.
\end{proof}

\begin{lemma}[Explicit computation of $A^k(q)$]
\label{ak}
For every $q>\varphi$ and for every $k\in\NN$
\begin{equation}\label{aq}
A^k(q)=\frac{1}{q^{k+1}}\left(
  \begin{array}{cc}
    f_{k+1} q & f_k \\
    f_k q^2 & f_{k-1}q \\
  \end{array}
\right).
\end{equation}
\end{lemma}

\begin{proof}
By induction on $k$. Base step, $k=1$, is trivially satisfied. Assume now (\ref{aq}) as inductive hypothesis. For $k+1$ we have
$$
A^{k+1}(q)=A^{k}(q) A(q)=\frac{1}{q^{k+2}}\begin{pmatrix}
                        (f_{k+1}+f_k) q & f_{k+1} \\
    (f_k + f_{k-1})q^2 & f_{k} \\
                      \end{pmatrix}=\frac{1}{q^{k+2}}\begin{pmatrix}
                        f_{k+2}q & f_{k+1} \\
    f_{k+1}q^2 & f_{k}q \end{pmatrix}.
$$
and this concludes the proof.
\end{proof}

\begin{proposition}
\label{kq1}
For every $q>\varphi$
\begin{equation}\label{kfirst}
k(q)\leq\frac{\ln\left(\frac{1}{\varphi^2 q^2}(q^4+3 q^2 +1)\right)}{2\left(\ln q - \ln \varphi\right)}.
\end{equation}
\end{proposition}
\begin{proof}
Fix $k$ and set
$$B(q):=\left(
  \begin{array}{cc}
    f_{k+1} q & f_k \\
    f_k q^2 & f_{k-1}q \\
  \end{array}
\right)$$
so that, by Lemma \ref{ak}, one has
$$A^k(q)=\frac{1}{q^{k+1}}B(q).
$$
Denote by $\lambda_{max}(A)$ the greatest eigenvalue of $A$ in modulus. One has that the matrix norm consistent with Euclidean norm satisfies the following identity
$$||A||:=\max_{\mathbf x\not=(0,0)}||A\mathbf x||=\sqrt{\lambda_{max}(A^T A)}.$$
Then
$$
||A^k(q)||=\frac{||B(q)||}{q^{k+1}}=\sqrt{\lambda_{max}(B^T(q) B(q))}.
$$
The product matrix $B^T(q) B(q)$ has the form:
$$
B^T(q) B(q)=\left(
  \begin{array}{cc}
    f_{k+1}^2 q^2+f_k^2 q^4 & f_k f_{k+1} q+f_{k}f_{k-1}q^3 \\
    f_k f_{k+1} q+f_{k}f_{k-1}q^3 & f_{k}^2+f_{k-1}^2 q^2 \\
  \end{array}
\right).
$$
The characteristic polynomial $p(\lambda)$ associated to $B^T(q) B(q)$ is hence
$$
\begin{array}{c}
  p(\lambda)=\lambda^2-\lambda\left(f_{k+1}^2 q^2+f_k^2 q^4+f_{k}^2+f_{k-1}^2 q^2\right)+ \\
  +q^4\left(f_{k-1}^2 f_{k+1}^2+f_k^4 - 2 f_{k+1}f_{k-1}f_k^2\right).
\end{array}
$$
The free term of characteristic polynomial is linked to algebraic identities involving the Fibonacci numbers,
$$
f_{k-1}^2 f_{k+1}^2+f_k^4 - 2 f_{k+1}f_{k-1}f_k^2=1.
$$
In fact
$$
f_{k-1}^2 f_{k+1}^2+f_k^4 - 2 f_{k+1}f_{k-1}f_k^2 =\left(f_k^2 - f_{k-1} f_{k+1}\right)^2
$$
involving Cassini's identity
$$f_{n-1}f_{n+1}-f_n^2=(-1)^{n+1},$$
Then, the characteristic polynomial becomes:
$$
p(\lambda)=\lambda^2-\lambda\left(f_{k+1}^2 q^2+f_k^2 q^4+f_{k}^2+f_{k-1}^2 q^2\right)+q^4.
$$
Set $\bar \lambda_{max}=f_k^2 q^4+(f_{k+1}^2+f_{k-1}^2) q^2+ f_{k}^2$ and note that
\begin{align*}
  \lambda_{max}(B^T(q)B(q))&= \frac{1}{2}\left(\bar \lambda_{max}+\sqrt{\lambda_{max}^2-4q^2}\right)\leq \bar \lambda_{max}.
\end{align*}
Furthermore by Proposition \ref{Fnphi} we have
$$
\bar{\lambda}_{max}\leq \varphi^{2k-2} q^4+(\varphi^{2k}+\varphi^{2k-4}) q^2+ \varphi^{2k-2}=\varphi^{2k-2}(q^4+3 q^2 +1)
$$
and finally
$$
||A^k(q)||=\frac{\lambda_{max}}{q^{2k+2}}\leq \frac{\bar{\lambda}_{max}}{q^{2k+2}}\leq  \frac{\varphi^{2k-2}}{q^{2k+2}}(q^4+3 q^2 +1).
$$
Consequently if
$$\frac{\varphi^{2k-2}}{q^{2k+2}}(q^4+3 q^2 +1)<1$$
then $||A^k(q)||<1$. To solve above inequality with respect to $k$ we apply the logarithm, requiring that the final report is less than  0:
$$
2k\ln\left(\frac{\varphi}{q}\right)+ \ln\left(\frac{1}{\varphi^2 q^2}(q^4+3 q^2 +1)\right)< 0.
$$
We finally obtain that if
$$
k>\frac{\ln\left(\frac{1}{\varphi^2 q^2}(q^4+3 q^2 +1)\right)}{2\left(\ln q - \ln \varphi\right)}
$$
then $||A^k(q)||<1$ and hence the claim.
\end{proof}
It is well-known that $f_k$ is the closest integer to ${\frac  {\varphi ^{k}}{{\sqrt  5}}}$.
Therefore it can be found by rounding in terms of the nearest integer function: $f_{k}={\bigg [}{\frac  {\varphi ^{k}}{{\sqrt  5}}}{\bigg ]},\ k\geq 0$.
That gives a very sharp inequality. In fact, if $k$ is an even number, then
$f_{k}={\bigg [}{\frac  {\varphi ^{k}}{{\sqrt  5}}}{\bigg ]}< \frac  {\varphi ^{k}}{{\sqrt  5}}$
i.e.
$f_{2k}={\bigg [}{\frac  {\varphi ^{2k}}{{\sqrt  5}}}{\bigg ]}< \frac  {\varphi ^{2k}}{{\sqrt  5}}$. We notice that $\frac  {\varphi ^{k}}{{\sqrt  5}}<\varphi ^{k-1}$. By the same procedure applied previously, we get
$$
\bar{\lambda}_{max}\leq \frac{q^2}{5}\left(\varphi^{2k-2}+\varphi^{2k+2}\right)+\frac{\varphi^{2k}}{5} (q^4+1).
$$
We have
$$
\begin{array}{c}
   \frac{\bar{\lambda}_{max}}{q^{2k+2}}\leq\left(\frac{\varphi}{q}\right)^{2k} \frac{1}{5q^2}\left(\frac{q^2}{\varphi^2}+q^2\varphi^2+q^4+1\right)\leq 1 \\
  \Leftrightarrow 2k\ln\left(\frac{\varphi}{q}\right)+ \ln\left(\frac{1}{5\varphi^2}+\frac{\varphi^2}{5}+\frac{q^2}{5}+\frac{1}{5 q^2}\right)\leq 0
\end{array}
$$
whence
\begin{equation}\label{ksecond}
k\geq\frac{\ln\left(\frac{1}{5\varphi^2}+\frac{\varphi^2}{5}+\frac{q^2}{5}+\frac{1}{5 q^2}\right)}{2\left(\ln q - \ln \varphi\right)} \ \ \ (q>\varphi)
\end{equation}
for $k$ even.

\begin{remark}
Now we want to compare the values ​​of $k(q)$, and suppose that $k(q)$ of (\ref{ksecond}) is greater than (\ref{kfirst}).
$$\frac{\ln\left(\frac{1}{5\varphi^2}+\frac{\varphi^2}{5}+\frac{q^2}{5}+\frac{1}{5 q^2}\right)}{2\left(\ln q - \ln \varphi\right)}>\frac{\ln\left(1+\frac{1}{\varphi^4}+\frac{q^2}{\varphi^2}+\frac{1}{q^2\varphi^2}\right)}{2\left(\ln q - \ln \varphi\right)}$$
i.e.
$$q^4(\varphi^4-5\varphi^2)+q^2 (\varphi^2+\varphi^6-5\varphi^4-5)-5\varphi^2+\varphi^4>0$$
which doesn't admit solution. Then
$$\frac{\ln\left(\frac{1}{5\varphi^2}+\frac{\varphi^2}{5}+\frac{q^2}{5}+\frac{1}{5 q^2}\right)}{2\left(\ln q - \ln \varphi\right)}<\frac{\ln\left(1+\frac{1}{\varphi^4}+\frac{q^2}{\varphi^2}+\frac{1}{q^2\varphi^2}\right)}{2\left(\ln q - \ln \varphi\right)}.$$
\end{remark}

\section{Conclusions}

We studied the workspace of a hyper-redundant manipulator, modeling a robot tentacle. We give a formal proof of the results, highlighted by numerical simulations based on a fractal geometry approach. The novelty of the paper consist to associate a linear control system involving Fibonacci sequence to the workspace of robot tentacle. We finally notice that, by the arbitrariness of the number of links, the asymptotic properties of the model (e.g. the possibility of setting an arbitrary global length for the manipulator)  extend by approximation to the case with a finite number of links with arbitrary small tolerance.

Chirikjian and Burdick's seminal report, \cite{term}, presented a kinematic algorithms for implementing planar hyper-redundant manipulator obstacle avoidance, and it suggests to us a further extension of this paper. In \cite{Li04} one has proposed to control the modularized hyper-redundant manipulators, obtaining the inverse kinematics solution of the planar hyper-redundant at the position and velocity levels; herein we can observe, for instance, the manipulator's configuration when the number of links is $4$ or $6$ and its length is $1$. Still we observe how, in this work, the length of the links is fixed, unlike approach showed here, where the lengths are controllable according to Fibonacci sequence.




\begin{thebibliography}{}



\bibitem[1]{snake}
V. V. Anderson, and R. C. Horn.
\newblock{Tensor-arm manipulator design.}
\newblock{\em {American Society of Mechanical Engineers,}} 67-DE-75: 1--12, 1967.


%
%

\bibitem[2]{Bai86}
 J. Baillieul.
\newblock Avoiding obstacles and resolving kinematic redundancy.
\newblock {\em  IEEE International Conference on Robotics and Automation,} 3: 1698--1704, 1986.
%


\bibitem[3]{Ball99}
Ball, Philip, and Neil R. Borley. The self-made tapestry: pattern formation in nature. \emph{Oxford: Oxford University Press}, 198, 1999.



\bibitem[4]{Bur88}
J. W. Burdick.
\newblock{Kinematic analysis and design of redundant robot manipulators.}
\newblock{{\em Diss. Stanford University}, 1988}.



\bibitem[5]{term}
 G. S. Chirikjian, and J. W. Burdick.
\newblock{An obstacle avoidance algorithm for hyper-redundant manipulators.}
\newblock{{\em IEEE International Conference on Robotics and Automation}, 1: 625--631, 1990.}

\bibitem[6]{CB95}
 G. S. Chirikjian, and J. W. Burdick.
\newblock{The kinematics of hyper-redundant robot locomotion.}
\newblock{{\em IEEE Transactions on Robotics and Automation,} 11.6: 781--793, 1995.}




\bibitem[7]{Cho99}
H. Choset and W. Henning.
\newblock A follow-the-leader approach to serpentine robot motion planning.
\newblock {\em J. Aerosp. Eng.}, 12(2): 65–-73, 1999.

\bibitem[8]{CP01}
Y.~Chitour and B.~Piccoli.
\newblock Controllability for discrete control systems with a finite control
  set.
\newblock {\em Mathematics of Control Signal and Systems}, 14: 173--193, 2001.






\bibitem[9]{IC96}
I. Ebert-Uphoff  and G. S. Chirikjian.
\newblock{Inverse kinematics of discretely actuated hyper-redundant manipulators using workspace densities.}
\newblock{{\em IEEE International Conference on Robotics and Automation}, 1: 149--145, 1996.}




\bibitem[10]{EK98}
P.~Erd{\"o}s and V.~Komornik.
\newblock Developments in non-integer bases.
\newblock {\em Acta Math. Hungar.}, 79(1-2): 57--83, 1998.

\bibitem[11]{Fal90}
K. Falconer.
\newblock Fractal geometry: mathematical foundations and applications.
\newblock{ \em John Wiley \and Sons}, 2013.




\bibitem[12]{Gil81}
W.~J. Gilbert.
\newblock Geometry of radix representations.
\newblock {\em The geometric vein}, 129--139. Springer, New York, 1981.

\bibitem[13]{Gil87}
W.~J. Gilbert.
\newblock Complex bases and fractal similarity.
\newblock {\em Ann. Sci. Math. Qu{\'e}bec}, 11.1: 65--77, 1987.

\bibitem[14]{IKR92}
K.-H. Indlekofer, I.~K{\'a}tai, and P.~Racsk{\'o}.
\newblock Number systems and fractal geometry.
\newblock {\em Probability theory and applications}, Springer Netherlands, 319--334, 1992.



\bibitem[15]{Hut81}
J. Hutchinson.
\newblock{Fractals and self-similarity}
\newblock{\em Indiana Univ. J. Math}, 30: 713--747, 1981.

\bibitem[16]{KL07}
V.~Komornik and P.~Loreti.
\newblock Expansions in complex bases.
\newblock {\em Canad. Math. Bull.}, 50.3: 399--408, 2007.



\bibitem[17]{Kwon91} S.J. Kwon, W.K.Chung, Y.Youm, and M.S.Kim. Self-collision avoidance for n-link redundant manipulators. \emph{Proceedings of the IEEE International Conference on System, Man and Cybernetics}, 937-942, Charlottesville, USA, October 1991.

\bibitem[18]{Lai11}
A.~C. Lai.
\newblock Geometrical aspects of expansions in complex bases.
\newblock {{\em Acta Mathematica Hungarica} 136.4: 275--300, 2012.}

\bibitem[19]{LL11}
A.~C. Lai and P.~Loreti.
\newblock Robot's finger and expansions in non-integer bases.
\newblock {\em Networks and Heterogeneus Media}, 7.1: 71--111, 2011.

\bibitem[20]{LL11a}
A.~C. Lai and P.~Loreti.
\newblock Robot's hand and expansions in non-integer bases.
\newblock {  Discrete Mathematics \and Theoretical Computer Science, 16(1), 371--394, 2014.}


\bibitem[21]{LL12}
A.~C. Lai and P.~Loreti.
\newblock Discrete asymptotic reachability via expansions in non-integer bases.
\newblock {\em Proceedings of 9-th International Conference on Informatics in Control, Automation and Robotics}, 2012.

\bibitem[22]{LLV14}
A.~C. Lai, P.~Loreti and P.~Vellucci.
\newblock  A model for Robotic Hand based on Fibonacci sequence.
\newblock {\em Proceedings of the 11-th International Conference on Informatics in Control, Automation and Robotics}, 577--587, 2014.


\bibitem[23]{Li04}
J. Liu, Y. Wang, S. Ma, B. Li.
\newblock Shape control of hyper-redundant modularized manipulator using variable structure regular polygon.
\newblock {\em Intelligent Robots and Systems,} 4: 3924--3929, 2004.




\bibitem[24]{Park03}
A. E. Park, J.J. Fernandez, K. Schmedders, M.S. Cohen. The Fibonacci sequence: relationship to the human hand, \emph{J. Hand Surg. Am.} 28 (1): 157-160, 2003.

%

\bibitem[25]{Par60}
W.~Parry.
\newblock On the {$\beta $}-expansions of real numbers.
\newblock {\em Acta Math. Acad. Sci. Hungar.}, 11:401--416, 1960.

%
%

\bibitem[26]{Ren57}
A.~R{\'e}nyi.
\newblock Representations for real numbers and their ergodic properties.
\newblock {\em Acta Math. Acad. Sci. Hungar}, 8:477--493, 1957.





\bibitem[27]{Wi12}
J. Wille, Occurrence of Fibonacci numbers in development and structure of animal forms: Phylogenetic observations and epigenetic significance. \emph{Natural Science}, 4: 216-232, 2012.





\end{thebibliography}
\end{document}